\documentclass[final,onefignum,onetabnum]{siamart190516}

\newif\ifmarkup
\markupfalse

\newcommand{\markupadd}[1]{%
\ifmarkup
\textcolor{blue}{#1}%
\else
#1%
\fi
}

\newcommand{\markupdelete}[1]{%
\ifmarkup
\sout{\textcolor{red}{#1}}%
\else
\fi
}

\usepackage[normalem]{ulem}
\usepackage{amsfonts}
\usepackage{graphicx}
\usepackage{epstopdf}
\usepackage{algorithmic}
\ifpdf
  \DeclareGraphicsExtensions{.eps,.pdf,.png,.jpg}
\else
  \DeclareGraphicsExtensions{.eps}
\fi

\usepackage{amsmath}
\usepackage{amssymb}
\usepackage{bm}
\usepackage{xcolor}
\usepackage{media9}
\usepackage{subfigure}
\usepackage{comment}
\usepackage{thmtools}
\usepackage{thm-restate}
\usepackage{cite}

\newcommand{\innp}[1]{\left\langle #1 \right\rangle}

\newcommand{\zeros}{\textbf{0}}
\newcommand{\vx}{\mathbf{x}}
\newcommand{\dd}{\mathrm{d}}
\newcommand{\cx}{\mathcal{X}}
\newcommand{\ch}{\mathcal{H}}
\newcommand{\chm}{\mathcal{H}_{\mathrm{M}}}
\newcommand{\cc}{\mathcal{C}}
\newcommand{\ccfsc}{\mathcal{C}^{f, \mathrm{sc}}}
\newcommand{\cz}{\mathcal{Z}}
\newcommand{\vxh}{\mathbf{\hat{x}}}

\newcommand{\vy}{\mathbf{y}}
\newcommand{\vz}{\mathbf{z}}
\newcommand{\vp}{\mathbf{p}}
\newcommand{\vq}{\mathbf{q}}
\newcommand{\vv}{\mathbf{v}}
\newcommand{\vw}{\mathbf{w}}

\newcommand{\vu}{\mathbf{u}}

\newcommand{\defeq}{\stackrel{\mathrm{\scriptscriptstyle def}}{=}}

\newcommand{\rr}{\mathbb{R}}

\DeclareMathOperator*{\argmin}{argmin}
\DeclareMathOperator*{\argmax}{argmax}

\newsiamremark{remark}{Remark}
\newsiamremark{hypothesis}{Hypothesis}
\crefname{hypothesis}{Hypothesis}{Hypotheses}
\newsiamthm{claim}{Claim}
\newsiamthm{fact}{Fact}

\headers{Generalized Momentum-Based Methods}{J.~Diakonikolas and M.~I.~Jordan}

\title{Generalized Momentum-Based Methods:\\ 
A Hamiltonian Perspective\thanks{
\funding{Research presented in this paper was partially  supported by the NSF grant \#CCF-1740855, by the Mathematical Data Science program of the
Office of Naval Research under grant number N00014-18-1-2764, and by the Office of the Vice
Chancellor for Research and Graduate Education at the University of Wisconsin–Madison with funding from the
Wisconsin Alumni Research Foundation. Part of this work was done while the authors were visiting Simons Institute for the Theory of Computing.}}}

\author{Jelena Diakonikolas\thanks{Department of Computer Sciences, UW-Madison, Madison, WI 
  (\email{jelena@cs.wisc.edu}).}
\and Michael I.~Jordan\thanks{Department of Statistics and EECS, UC Berkeley, Berkeley, CA 
  (\email{jordan@cs.berkeley.edu}).}
}

\usepackage{amsopn}

\AtBeginDocument{%
  \paperwidth=\dimexpr
    1in + \oddsidemargin
    + \textwidth
    + 1in + \oddsidemargin
  \relax
  \paperheight=\dimexpr
    1in + \topmargin
    + \headheight + \headsep
    + \textheight
    + 1in + \topmargin
  \relax
  \usepackage[pass]{geometry}\relax
}

\begin{document}

\maketitle
\begin{abstract}
    We take a Hamiltonian-based perspective to generalize Nesterov's accelerated gradient descent and Polyak's heavy ball method to a broad class of momentum methods in the setting  of (possibly) constrained minimization in \markupdelete{Banach} \markupadd{Euclidean and non-Euclidean normed vector spaces}. Our perspective leads to a generic and unifying nonasymptotic analysis of convergence of these methods in both the function value (in the setting of convex optimization) and in norm of the gradient (in the setting of unconstrained, possibly nonconvex, optimization). Our approach relies upon a time-varying Hamiltonian that produces generalized momentum methods as its equations of motion. 
    The convergence analysis for these methods is intuitive and is based on the conserved quantities of the time-dependent Hamiltonian. 
\end{abstract}

\begin{keywords}
  acceleration, momentum-based methods, stationary points, Hamiltonian dynamics
\end{keywords}


%
%
\section{Introduction}

Accelerated, momentum-based, methods enjoy optimal iteration complexity for the minimization of  smooth convex functions over convex sets, which has led to their broad acceptance as algorithmic primitives in many applications, notably applications in machine learning. Further, decades of empirical experience suggest that momentum methods are capable of exploring multiple local minima (see, e.g.,~\cite{attouch2000heavy-esc,bostan2018accelerated,xu2018accelerated}), which gives them advantages over gradient flows.  The latter have optimal worst-case complexity for convergence to stationary points, but are strongly attracted to local minima. Moreover, recent theoretical results have established that momentum methods escape saddle points faster than standard gradient descent~\cite{jin2017accelerated,oneil2017behavior}, providing further evidence of their value in nonconvex optimization.

The first (locally) accelerated method for smooth and strongly convex minimization\footnote{In the special case of quadratic objectives, Polyak's method attains the globally-accelerated (iteration-complexity optimal) convergence rate.} is from the 1960s and is due to Polyak~\cite{polyak1964some}. Working with continuous-time dynamics, Polyak introduced the following (momentum-based) second-order ordinary differential equation (ODE):
\begin{equation}\tag{HBD}\label{eq:HBD}
    \ddot{\vx}_t  = \alpha_1 \dot{\vx}_t + \alpha_2 \nabla f(\vx_t),
\end{equation}
where $f$ is the function being minimized and $\alpha_1,\, \alpha_2$ are constants. He also studied its two-step discretization, which can be written as:
\begin{equation}\label{eq:HB}\tag{HB}
    \vx_{k+1} = \vx_k -\alpha \nabla f(\vx_k) + \beta (\vx_k - \vx_{k-1}),
\end{equation}
where $\alpha,\, \beta$ are constants. In particular, under a suitable choice of $\alpha,\, \beta$, \cite{polyak1964some} showed that, when initialized ``sufficiently close'' to the optimal solution $\vx^*,$ the method converges at rate $(\frac{1-\sqrt{\kappa}}{1+\sqrt{\kappa}})^k$ 
where $\kappa$ denotes the condition number of the objective function $f$. This convergence rate was later proved to be optimal and globally achievable~\cite{nemirovskii1983problem}.

For the setting of smooth minimization, Nesterov~\cite{Nesterov1983} discovered a method that has an optimal convergence rate of $1/k^2.$ Nesterov~\cite{Nesterov1983} also showed that when this method is applied to strongly convex functions with a known strong convexity parameter and coupled with scheduled restart, it leads to the optimal rate for the class of smooth and strongly convex minimization problems. In later work (see, e.g., the textbook~\cite{nesterov2013introductory} and references therein), Nesterov introduced a separate, more direct method for the minimization of smooth and strongly convex functions that enjoys the same optimal rate as~\eqref{eq:HB} but \markupadd{for which the optimal rate holds globally while the algorithm has} \markupdelete{that holds globally and has} the same continuous-time limit~\eqref{eq:HBD}.

A flurry of research has followed these seminal papers on momentum-based methods~\cite{tseng2008,nemirovski1985optimal,AO-survey-nesterov,attouch2000heavy,attouch2019rate,SuBC16,wibisono2016variational,wilson2016lyapunov,zhang2018direct,shi2018understanding,krichene2015accelerated,Scieur2017,Bubeck2015,drusvyatskiy2016optimal,lin2015universal,betancourt2018symplectic,AXGD,thegaptechnique,attouch2000heavy-esc,lessard2019direct,hu2017control,muehlebach2019dynamical}, with many of these works~\cite{attouch2019rate,SuBC16,wibisono2016variational,wilson2016lyapunov,zhang2018direct,shi2018understanding,krichene2015accelerated,Scieur2017,betancourt2018symplectic,AXGD,thegaptechnique,muehlebach2019dynamical,francca2019conformal,muehlebach2019dynamical} seeking to interpret Nesterov acceleration as a discretization of a continuous-time dynamical system. Further, some of these works led to physical interpretations of Polyak's~\cite{attouch2000heavy,attouch2000heavy-esc,francca2019conformal} and Nesterov's~\cite{wibisono2016variational,betancourt2018symplectic} methods in the Lagrangian and Hamiltonian formalism, and the physical interpretation of momentum has even led to new algorithms (e.g.,~\cite{maddison2018hamiltonian,francca2019conformal,betancourt2018symplectic}). For the setting of nonconvex optimization, however, and, more broadly, convergence to points with small norm of the gradient, the continuous-time perspective has been much less explored~\cite{attouch2000heavy-esc,shi2018understanding,jin2017accelerated,attouch2019first}.

In this paper, we take a Hamiltonian-based perspective to derive a broad class of momentum methods that yield Nesterov's and Polyak's methods as special cases. 
As a specific example, a class of methods obtained from the introduced Hamiltonian and parametrized by $\lambda \in [0, 1]$ interpolates between Nesterov's method for smooth minimization~\cite{Nesterov1983} (when $\lambda = 1$) and a generalization of the heavy ball method~\cite{polyak1964some} (when $\lambda = 0$). 
We show that because the methods are obtained as the equations of motion of this Hamiltonian, we can deduce invariants (conserved quantities of the Hamiltonian) that can be used to argue about convergence in function value (for convex optimization) and convergence to stationary points (for possibly nonconvex optimization). The techniques are general and lead to results in \markupdelete{Banach} \markupadd{general normed vector} spaces. 

\markupadd{We note that non-Euclidean normed spaces are frequently encountered in applications, particularly those arising in sparsity-oriented machine learning, where it is natural to use the $\ell_1$ norm, and in fast algorithms for network flow problems, where the natural norm is $\ell_\infty$~\cite{KLOS2014,Lee:2013,sherman2017area}. Here, by ``natural,'' we mean that the use of an alternative $\ell_p$ norm would incur a polynomial dependence on the dimension in the method's iteration complexity. This is because both the initial distance to the optimum or the diameter of the set \emph{and} the Lipschitz constant of the objective function or its gradient that determine the iteration complexity are defined w.r.t.~the norm that is adopted for the underlying vector space (see, e.g.,~\cite{diakonikolas2020lower} for a similar discussion).}

In Section~\ref{appx:conv-in-fun-val} we provide analysis of the methods in terms of convergence in function value in the setting of (possibly) constrained convex optimization in \markupdelete{Banach} \markupadd{possibly non-Euclidean normed vector} spaces. We show that the entire class of methods parametrized by $\lambda$ (as mentioned above) converges at rate $1/k^2$ as long as $\lambda$ is bounded away from zero. When $\lambda = 0,$ the convergence slows down to $1/k.$ This agrees with previously obtained results for the heavy-ball method in the setting of smooth (non-strongly convex) minimization~\cite{ghadimi2015global} (our case $\lambda = 0$). As a byproduct of this approach, we obtain a generalization of the heavy-ball method to constrained convex optimization in \markupdelete{Banach} \markupadd{non-Euclidean vector} spaces and show that it converges at rate $1/k.$ Such a result was previously known only for unconstrained convex optimization in Euclidean spaces~\cite{ghadimi2015global}.

In terms of the convergence to stationary points, we consider the unconstrained case and focus on finding points with small norm of the gradient, in either \markupdelete{Hilbert or} \markupdelete{Banach} \markupadd{Euclidean or non-Euclidean} spaces. 
We show that when $f$ is convex, any method from the class satisfies $\min_{0\leq i\leq k}\|\nabla f(\vx_k)\|_*^2 = O(\frac{L(f(\vx_0 - \vx^*))}{k}).$ Note that any of these methods, when run for $k/2$ iterations after running Nesterov's method for $k/2$ iterations, satisfies $\min_{k/2\leq i\leq k}\|\nabla f(\vx_k)\|_*^2 = O(\frac{L\|\vx_0 - \vx^*\|^2}{k^3}).$ While this is suboptimal for the case of convex functions---the optimal rate is $\min_{0\leq i\leq k}\|\nabla f(\vx_k)\|_*^2 = O(\frac{L(f(\vx_0 - \vx^*))}{k^2})$~\cite{nesterov2012make,kim2018optimizing,carmon2017lower} and it is achieved by~\cite{kim2018optimizing}---we conjecture that it is tight. In particular,~\cite{kim2018generalizing} demonstrated that the convergence of the form $\min_{k/2\leq i\leq k}\|\nabla f(\vx_k)\|_*^2 = O(\frac{L\|\vx_0 - \vx^*\|^2}{k^3})$ is tight for Nesterov's method. 

For the case of nonconvex functions, we show that methods that are instantiations of the heavy-ball method converge at the optimal~\cite{carmon2017lower} rate of $\min_{0\leq i\leq k}\|\nabla f(\vx_k)\|_*^2 = O(\frac{L(f(\vx_0 - \vx^*))}{k}).$ While a similar result exists for the case of Euclidean spaces~\cite{ghadimi2016accelerated}, we are not aware of any other results for the more general \markupdelete{Banach} \markupadd{normed vector} spaces in which the method does not lose its favorable properties. For example, it is possible to establish similar rates for modifications of Nesterov's method that turn it into a descent method that makes at least as much progress as gradient descent, as in, e.g.,~\cite{nesterov2012make,nesterov2018primal,Ghadimi2019}. In this case, the analysis of convergence in norm of the gradient boils down to the analysis of (standard) gradient descent. Unfortunately, because such methods monotonically decrease the function value, they lose the property of utilizing the momentum to escape shallow local minima. Note that, as mentioned earlier, global exploration of local minima is one of the primary reasons for considering momentum-based methods in nonconvex optimization~\cite{attouch2000heavy-esc}.

{Finally, we note that the notions of convergence considered in this paper are only the standard notions for first-order smooth (not necessarily strongly convex) optimization: (i) convergence to (neighborhoods of) points with  small optimality gap, and (ii) convergence to (neighborhoods of) points with  small gradient norm. Our analysis only guarantees that such points can be output by the algorithm and does not deal with the convergence of the trajectories, which is \markupadd{generally} a \markupdelete{much} more challenging problem (see, e.g.,~\cite{cabot2009long}). \markupdelete{In particular, we do not rule out the possibility that the dynamics}  \markupdelete{oscillate around stationary points.}} \markupadd{We note that weak convergence results of the trajectories have been established by prior work for variants of inertial proximal algorithms that generalize the Nesterov accelerated method to the setting of composite objectives (smooth plus nonsmooth)~\cite{chambolle2015convergence,attouch2019convergence}.}
\subsection{Related Work}
In addition to the work already mentioned above, we provide a few more remarks regarding related work. First, for convex minimization (in function value), there are several approaches that apply to constrained minimization and general \markupdelete{Banach} \markupadd{normed vector} spaces~\cite{wibisono2016variational,krichene2015accelerated,AXGD,thegaptechnique,betancourt2018symplectic}, with a subset of them being directly motivated by Lagrangian~\cite{wibisono2016variational} and Hamiltonian~\cite{betancourt2018symplectic} mechanics. The latter make use of \markupadd{special} Lyapunov functions \markupadd{(which are sometimes also called the energy functions)} to characterize convergence rates, and their applicability to convergence to stationary points is unclear. In contrast, our work is not based on separately constructed Lyapunov functions; rather, our analysis of convergence rates stems from the analysis of conserved quantities of the Hamiltonian. \markupadd{These conserved quantities can also be viewed as Lyapunov functions. The main difference compared to the prior work is that we do not need to ``guess'' these functions---they are directly derived from the same Hamiltonian whose equations of motion are the momentum dynamics we consider.}   

A significant body of recent work in nonconvex optimization has focused on convergence to approximate local minima, with many of the methods having (near-)optimal iteration complexities (see, e.g.,~\cite{carmon2017lower,allen2017natasha,jin2017accelerated}). The only work that we are aware of that has used a Hamiltonian perspective on convergence to stationary points in the nonconvex setting is~\cite{jin2017accelerated}. However, the connection to Hamiltonian systems in~\cite{jin2017accelerated} is limited---it essentially relies on showing that~\eqref{eq:HBD} dissipates energy of the form $f(\vx_t) + \frac{1}{2}\|\dot{\vx}_t\|_2^2,$ and is specialized to a Euclidean setting. In fact, the main contribution of~\cite{jin2017accelerated} is in providing a near-optimal method for convergence to approximate local minima and not in providing a Hamiltonian perspective on nonconvex optimization. \markupadd{Also worth mentioning are inertial approaches to nonconvex optimization~\cite{begout2015damped, castera2019inertial}, based on Kurdyka-{\L}ojasiewicz (KL) property of the objective, as introduced in~\cite{bolte2010characterizations}. The KL property ensures that the gradients of the objective do not vanish too quickly around critical points. Such a property is not assumed to hold for the problems considered in this work.}
%
%
\subsection{Preliminaries}

%
\markupadd{While there is nothing in our approach that prevents one from working in general (possibly infinitely-dimensional) Banach spaces, we will limit our attention to finite-dimensional real vector spaces, as they suffice for the motivating applications discussed in the introduction. }
The primal, $n$-dimensional real vector space is denoted by $E.$ The space $E$ is normed, endowed with a norm $\|\cdot\|.$ Its dual space, consisting of all linear functions on $E,$ is denoted by $E^*.$ For $\vz \in E^*$ and $\vx \in E,$ we denote by $\innp{\vz, \vx}$ the value of $\vz$ at $\vx.$ The  dual norm (associated with space $E^*$) is defined in the standard way as $\|\vz\|_* = \max_{\vx \in E}\frac{\innp{\vz, \vx}}{\|\vx\|}.$ For Euclidean spaces, $\innp{\cdot, \cdot}$ is the standard inner product and $\|\cdot\| = \|\cdot\|_* = \|\cdot\|_2.$ 

We assume that $f: \cx \rightarrow \mathbb{R}$ is a (possibly nonconvex) continuously-differentiable function, and $\cx \subseteq E$ is closed and convex. 
$\vx^* \in \argmin_{\vx \in \cx} f(\vx)$ denotes any fixed minimizer of $f$. To avoid vacuous statements, we assume that $f(\vx^*) > - \infty.$ 


For all the methods, $\vx_t \in \cx$ will be the running solution, and $\vz_t \in \cz = \vz_0 + \mathrm{Lin}\{\nabla f(\vx): \vx \in \cx\}$ 
will be \markupadd{the sum of $\vz_0 \in E^*$ and} some linear combination of the gradients $\nabla f(\vx_{\tau})$ for $\tau \in [0, t].$ In the Hamiltonian formalism, $\vz_t$ will correspond to the \markupdelete{momentum} \markupadd{conjugate momenta}, $f(\vx)$ will correspond to the potential energy, and $\psi^*(\vz)$ will correspond to the kinetic energy, where $\psi^*:\cz\rightarrow \mathbb{R}$ is a convex conjugate (defined below) of some strongly convex function $\psi:\cx \rightarrow \mathbb{R}$ (e.g., $\psi^* (\vz) = \frac{1}{2}\|\vz\|_2^2$ if $\|\cdot\| = \|\cdot\|_2$ and $\psi(\vx) = \frac{1}{2}\|\vx\|_2^2$)\footnote{Here we use the notation $\psi^*$ to emphasize that $\vx$ and $\vz$ do not, in general, belong to the same vector space.}. 

We now outline some useful definitions and facts that are used in the paper.

\markupdelete{The equations of motion corresponding to a Hamiltonian $\ch(\vx, \vz, t)$ are given by\\}
\markupdelete{$\dot{\vx}_t = \nabla_{\vz}\ch(\vx, \vz, t)$ and $\dot{\vz}_t = -\nabla_{\vx}\ch(\vx, \vz, t),$ where $\nabla_{\vx}\ch(\vx, \vz, t)$ (respectively, \\}
\markupdelete{$\nabla_{\vz}\ch(\vx, \vz, t)$) denotes the partial gradient of $\ch(\vx, \vz, t)$ w.r.t.~$\vx$ (respectively, $\vz$).\\} \markupdelete{Differentiating $\ch(\vx, \vz, t)$ 
with respect to time $t,$ if $\vx,\, \vz$ evolve according to the \\}
\markupdelete{equations of motion of $\ch(\vx, \vz, t)$, it follows that $\frac{\dd }{\dd t}\ch(\vx, \vz, t) = \frac{\partial}{\partial t}\ch(\vx, \vz, t).$ This will\\} 
\markupdelete{be used to derive the conserved quantities (invariants) of the generalized momentum\\}
\markupdelete{Hamiltonian in the proof of Lemma~\ref{lemma:ct-gen-mom-cq} from Section~\ref{sec:ct-methods}.}

To carry out the analysis of the cases of convex and nonconvex objectives in a unified way, we use the following notion of weak convexity, similar to e.g.,~\cite{allen2017natasha,davis2019proximally}.
\begin{definition}\label{def:eps-weakly-ncvx}
We say that a continuously-differentiable function $f: \cx \rightarrow \rr$ is $\epsilon_H$-weakly convex for some $\epsilon_H \in \rr_+$ if 
$$(\forall \vx, \vy \in \cx):\quad f(\vy) \geq f(\vx) + \innp{\nabla f(\vx), \vy - \vx} - \frac{\epsilon_H}{2}\|\vy - \vx\|^2.$$
\end{definition}
We will mainly be concerned with cases $\epsilon_H = 0$ and $\epsilon_H = L.$ Observe that a $0$-weakly convex function is convex, by the standard first-order definition of convexity for continuously-differentiable functions that can be stated as 
$$(\forall \vx,\, \vy \in \cx):\quad f(\vy) \geq f(\vx) + \innp{\nabla f(\vx), \vy - \vx}.$$ 

\begin{definition}\label{def:smoothness}
A continuously-differentiable function $f: \cx \rightarrow \rr$ is $L$-smooth for $L \in \rr_+,$ if $\forall \vx, \vy \in \cx,$ $\|\nabla f(\vx) - \nabla f(\vy)\|_* \leq L\|\vx - \vy\|.$
\end{definition}
Recall that $L$-smoothness of a function implies that 
$$(\forall \vx, \vy \in \cx):\quad f(\vy) \leq f(\vx) + \innp{\nabla f(\vx), \vy - \vx} + \frac{L}{2}\|\vy - \vx\|^2.$$ 
It is not hard to show that an $L$-smooth function is also $L$-weakly convex. We will be assuming throughout that there exists $L < \infty$ such that $f$ is $L$-smooth. 

\begin{definition}\label{def:strong-cvxity}
A continuously-differentiable function $f: \cx \rightarrow \rr$ is $\mu$-strongly convex for $\mu \in \rr_+,$ if $\forall \vx, \vy \in \cx,$ $f(\vy) \geq f(\vx) + \innp{\nabla f(\vx), \vy - \vx} + \frac{\mu}{2}\|\vy - \vx\|^2.$
\end{definition}

\begin{definition}\label{def:cvx-conj}
The convex conjugate of $\psi : \cx \rightarrow \mathbb{R}$ is defined as $\psi^*(\vz) = \sup_{\vx \in \cx} \{\innp{\vz, \vx} - \psi(\vx)\}$.\markupadd{\footnote{\markupadd{Our definition of a convex conjugate slightly differs from the standard definition, where the convex conjugate is defined w.r.t.~the entire vector space $E,$ i.e., $\psi^*(\vz) = \sup_{\vx \in E} \{\innp{\vz, \vx} - \psi(\vx)\}$. Thus, our definition corresponds to taking the convex conjugate of $\psi + I_{\cx}$ in the standard definition, where $I_{\cx}$ is the indicator function of the set $\cx$ (equal to zero within the set, and to infinity outside it). This choice incurs no loss of generality and simplifies the notation throughout the paper.}}}
\end{definition}

The following standard fact is a corollary of Danskin's Theorem (see, e.g.,~\cite{bertsekas1971control}).
\begin{fact}\label{fact:danskin}
Let $\psi:\cx \rightarrow \mathbb{R}$ be a strongly convex function.\footnote{Alternatively, it suffices that $\psi$ is strictly convex and \markupadd{either} $\cx$ is compact \markupadd{or the gradient of $\psi$ blows up at the boundary of $\cx$.}} Then $\psi^*$ is continuously differentiable and $\nabla \psi^*(\vz) = \argmax_{\vx \in \cx}\{\innp{\vz, \vx} - \psi(\vx)\}$.
\end{fact}

Another useful property of convex conjugacy is the duality between smoothness and strong convexity, which can be seen as a strengthening of Fact~\ref{fact:danskin}.

\begin{fact}\label{fact:smoothness-sc-duality}
Let $\psi:\cx \rightarrow \mathbb{R}$ be a $\mu$-strongly convex function. Then $\psi^*$ is $\frac{1}{\mu}$-smooth.
\end{fact}
For a convex function $\psi:\vx \rightarrow \mathbb{R}$ that is continuously differentiable on $\cx$, Bregman divergence is defined in a usual way as $D_\psi(\vy, \vx) = \psi(\vy) - \psi(\vx) - \innp{\nabla \psi(\vx), \vy - \vx},$ where $\vx, \vy \in \cx$ . Some useful properties of Bregman divergence are stated below.
\begin{fact}\label{fact:bregman-properties}
(Properties of Bregman Divergence.) Let $\psi:\cx \rightarrow \mathbb{R}$ be  convex and continuously differentiable on $\cx.$ Then:
\begin{itemize}
    \item[(i)] For any $\vu, \vv, \vw \in \cx,$ the following three-point identity holds:
    $$D_\psi(\vu, \vv) = D_{\psi}(\vw, \vv) + \innp{\nabla \psi(\vw) - \nabla\psi(\vv), \vu - \vw} + D_{\psi}(\vu, \vw).$$
    \item[(ii)] Let $\psi$ be $\mu$-strongly convex w.r.t.~$\|\cdot\|$. Then, $\forall \vz, \vz',$ $$D_{\psi^*}(\vz, \vz') \geq \frac{\mu}{2}\|\nabla\psi^*(\vz) - \nabla\psi^*(\vz')\|^2.$$ 
\end{itemize}
\end{fact}
%
%
%
%
%
\section{Continuous-Time Methods and their  \markupadd{Convergence Analysis}}\label{sec:ct-methods}

We start this section by \markupadd{providing a brief overview of the well-studied inertial approach that has been widely used in continuous optimization to study the behavior of momentum-based methods. We then provide basic definitions and illustrate the main ideas that underlie our analysis. In doing so, we compare our approach to the inertial one and highlight the connections between them.} \markupdelete{describing a simple dynamical system} \markupdelete{obtained from a time-invariant Hamiltonian. Such a  dynamics is known to be non-\\}
\markupdelete{convergent---because the total energy of the system is conserved,  there is a continual} 
\markupdelete{exchange between potential and kinetic energy. 
We thus replace this non-convergent\\} 
\markupdelete{Hamiltonian dynamics with a dynamical system obtained from a time-varying\\} 
\markupdelete{Hamiltonian that dissipates energy. We then show that conserved quantities of this\\} 
\markupdelete{new Hamiltonian can be used to argue about convergence to neighborhoods of\\} 
\markupdelete{stationary points. In later sections, we use the discrete versions of these conserved\\} 
\markupdelete{quantities to argue about convergence in function value (for convex optimization)\\}  
\markupdelete{and convergence to stationary points (for potentially nonconvex optimization).}

\subsection{\markupadd{Inertial Approach in Optimization}}

\markupadd{One of the earliest examples of the inertial approach is the heavy-ball method introduced by Polyak~\cite{polyak1964some} and defined in \eqref{eq:HBD}, \eqref{eq:HB}. The method~\eqref{eq:HBD} applies to \emph{unconstrained Euclidean settings}, and its generalizations to Hilbert spaces and constrained setups has been studied in~\cite{attouch2000heavy-esc,attouch2000heavy,attouch2002dynamics,attouch2012second,laraki2015inertial}.}

\markupadd{The basic variant of the second-order differential equation considered as part of the (damped) inertial approach can be stated as:}
\begin{equation}\label{eq:inertial-ode}
   \markupadd{ \ddot{\vx}_t + \alpha_1(t) \dot{\vx}_t + \alpha_2(t) \nabla f(\vx_t) = 0,}
\end{equation}
\markupadd{where $\ddot{\vx}_t$ is the inertial term, $\alpha_1(t) \dot{\vx}_t$ is the friction term that dissipates energy, and $\alpha_2(t) \nabla f(\vx_t)$ is the term that corresponds to potential forces that drive the motion.} 
\markupadd{When $\alpha_1(t) = \alpha_1$ and $\beta_1(t) = \beta_1$ are constants, Eq.~\eqref{eq:inertial-ode} reduces to Polyak's heavy-ball method as stated in~\eqref{eq:HBD}. When $\alpha_1(t) = \frac{\alpha}{t}$ for $\alpha >0$ and $\alpha_2(t)=1,$ Eq.~\eqref{eq:inertial-ode} reduces to the differential equation introduced by Su et al.~\cite{SuBC16} for studying Nesterov acceleration and its generalizations. In particular, \cite{SuBC16} showed that Nesterov acceleration corresponds to the case $\alpha = 3,$ while choices of $\alpha \geq 3$ can generally lead to accelerated $1/t^2$ rates for smooth minimization. The subcritical case $\alpha <3$ and the degradation of the rates from $1/t^2$ to $1/t^{\frac{2}{3}\alpha}$ was characterized by Attouch et al.~\cite{attouch2019rate}.}

\markupadd{While the inertial approach is powerful and can lead to different qualitative results (e.g., it can be used to study convergence of the trajectories that is not considered in this work), its main limitation is that technical obstacles are encountered once generalizations to non-Euclidean and constrained setups are considered. }

\markupadd{In its basic form, the dynamics from Eq.~\eqref{eq:inertial-ode} cannot be directly applied to  non-Euclidean setups, as the points $\vx_t$ and the gradients $\nabla f(\vx_t)$ do not in general belong to the same vector space. This issue was resolved in the work by Wibisono et al.~\cite{wibisono2016variational} using the concept of Bregman Lagrangian and the dynamics that are generated from it. However, it is not immediately clear how to generalize the results from \cite{wibisono2016variational} to constrained setups.}

\markupadd{A specific challenge that is encountered when imposing constraints directly in the second order ODE from Eq.~\eqref{eq:inertial-ode} is that the direct enforcement of the constraints typically involves the use of maps that are not differentiable and which turn differential equation problems into differential inclusion problems (see, e.g., \cite{attouch2002dynamics,attouch2012second}). More importantly, the introduction of the constraints within this inertial approach is well-known to cause non-elastic shocks, due to the discontinuities in the velocity $\dot{\vx}$ encountered at the boundary of the feasible region~\cite{attouch2002dynamics,attouch2012second}. This problem is overcome in our work (and similarly in prior work~\cite{krichene2015accelerated,thegaptechnique}) through the use of constraint regularization.} 

\markupadd{Finally, the convergence analysis in the inertial approach is carried out using Lyapunov functions that have the interpretation of total energy of the system. They are typically constructed as the sum of the (possibly) scaled optimality gap $f(\vx) - f(\vx^*)$ and another function that has the interpretation of kinetic energy. While the energy interpretation is intuitive for these Lyapunov functions, it is not always \emph{a priori} clear how to construct them. In our approach, it is the Hamiltonian itself that has the interpretation of the total energy, and both our Lyapunov functions and the continuous-time dynamics are obtained from this same Hamiltonian.}

\subsection{\markupadd{Background on Hamiltonian Mechanics}}
\markupadd{In Hamiltonian mechanics, the motion of a particle, or, more broadly, a physical system, is described by canonical coordinates $(\vx, \vz)$\footnote{\markupadd{In physics, it is standard to use the notation $(\vp, \vq)$ instead to denote the generalized coordinates and the conjugate momenta. We use $(\vx, \vz)$ for consistency with the optimization literature.}}, where $\vx$ corresponds to the generalized coordinates (position of a particle) and $\vz$ are their conjugate momenta. Given a Hamiltonian $\ch(\vx, \vz, t),$  which is typically interpreted as the total energy of the system, the time evolution of the physical system is described by Hamilton's equations:}
\begin{equation}\label{eq:hamilton-eq}
   \markupadd{ \frac{\dd \vx}{\dd t} = \nabla_{\vz} \ch(\vx, \vz, t), \quad \frac{\dd \vz}{\dd t} = - \nabla_{\vx} \ch(\vx, \vz, t). }
\end{equation}
\markupadd{An immediate implication of Eq.~\eqref{eq:hamilton-eq} that will be used in deriving the conserved quantities used in the convergence analysis is the following key relationship:}
\begin{equation}\label{eq:hamitonian-time-derivative}
    \begin{aligned}
        \markupadd{ \frac{\dd \ch(\vx, \vz, t)}{\dd t}} &\markupadd{= \nabla_{\vx} \ch(\vx, \vz, t)\cdot \frac{\dd \vx}{\dd t} + \nabla_{\vz} \ch(\vx, \vz, t) \cdot \frac{\dd \vz}{\dd t} + \frac{\partial \ch(\vx, \vz, t)}{\partial t}}\\
         &\markupadd{= \frac{\partial \ch(\vx, \vz, t)}{\partial t}.}
    \end{aligned}
\end{equation}

Perhaps the simplest Hamiltonian that one can formulate is the following separable function:
$
\ch({\vx}, \vz) = f(\vx) + \psi^*(\vz).
$ 
Here, $f(\vx)$ can be viewed as the potential energy of a particle at position $\vx,$ while $\psi^*(\vz)$ is its kinetic energy. 
The corresponding continuous-time dynamics is:
\begin{equation}\label{eq:ct-simple-dyn}\tag{HD}
\begin{gathered}
    \dot{\vx}_t = \nabla_{\vz_t}\ch({\vx}_t, \vz_t) = \nabla \psi^*(\vz_t),\\ 
    \dot{\vz}_t = - \nabla_{\vx_t}\ch({\vx}_t, \vz_t) = -\nabla f(\vx_t).
\end{gathered}
\end{equation}
This dynamics is meaningful only in the unconstrained regime, since otherwise we cannot guarantee that $\vx_t\in \cx.$ Hence, we assume here that $\cx = \mathbb{R}^n.$ 
As $\ch({\vx}, \vz)$ does not explicitly depend on time, we have $\frac{\dd}{\dd t}\ch({\vx}, \vz) = 0.$ Equivalently, $\ch({\vx}, \vz)$ is conserved with time. An immediate implication is that the norm of the averaged gradient decays as $1/t,$ as stated in the following proposition.
\begin{proposition}\label{lemma:simple-ham-dyn-conv}
Let $\vx_t, \vz_t$ evolve according to~\eqref{eq:ct-simple-dyn}, for $\vz_{0} = \zeros$ and arbitrary (but fixed) $\vx_{0}\in \mathbb{R}^n,$ and let $\psi^*(\zeros) = 0.$ If $\psi^*$ is $\mu$-strongly convex,  then, $\forall t \geq 0:$
$$
\left\|\frac{1}{t}\int_0^t \nabla f(\vx_{\tau})\dd \tau \right\|_* \leq \frac{\sqrt{2(f(\vx_{0})-f(\vx^*))/\mu}}{t}.
$$
\end{proposition}
\begin{proof}
Since the Hamiltonian is conserved, $\psi^*(\vz_t) = f(\vx_{0})-f(\vx_t) + \psi^*(\vz_{0}) = f(\vx_{0})-f(\vx_t).$ \markupadd{Recall that $\vx^*$ minimizes $f.$} As $f(\vx_t) \geq f(\vx^*),$ it follows that $\psi^*(\vz_t) \leq f(\vx_{0})-f(\vx^{*}).$ By the $\mu$-strong convexity of $\psi^*,$ $\frac{\mu}{2}\|\vz_t\|_*\leq f(\vx_{0})-f(\vx^{*}).$ Finally, integrating the second equation from~\eqref{eq:ct-simple-dyn}, we have that $\vz_t = \vz_{0}-\int_0^t \nabla f(\vx_{\tau})\dd\tau = -\int_0^t \nabla f(\vx_{\tau})\dd\tau$, which, combined with the last inequality (after dividing both sides by $t^2\mu/2$ and taking the square root of both sides), gives the claimed bound.
\end{proof}

While Proposition~\ref{lemma:simple-ham-dyn-conv} shows that the \emph{average} of the gradients converges in norm $\|\cdot\|_*$ to zero at a sublinear rate, it does not guarantee that the dynamics converges to or even visits any stationary point of $f$. Indeed, since the energy (equal to $\ch({\vx}_t, \vz_t)$) is conserved with time, the dynamics is well-known to be non-convergent. Hence, the fact that the average gradient converges in the dual norm only implies that the path of the dynamics consists of cycle-like segments over which the gradients cancel out. 

%
%
\subsection{Generalized Momentum Dynamics}
%
%
%
 %
 %
The standard Hamiltonian dynamics from the previous subsection is overly aggressive as a function of the history of the gradients (i.e., the momentum $\vz_t$). As a consequence of energy conservation, the energy is exchanged between the potential and kinetic energy, which makes the dynamics exhibit non-convergent behavior. For the dynamics to be attracted to stationary points, it needs to be dampened.  The most common approach is to introduce friction into the equations of motion, which leads to second order ODEs such as those described by Eq.~\eqref{eq:inertial-ode}. 
Here, we take an alternative approach that directly modifies the Hamiltonian. As we will see, this approach allows us to consider  constrained optimization problems over general \markupdelete{Banach}\markupadd{normed} spaces, unlike \markupadd{most of} the friction-based approaches.

For example, the Nesterov acceleration method for smooth constrained minimization in general normed spaces can be expressed in continuous time as follows~\cite{krichene2015accelerated,thegaptechnique}:
\begin{equation}\label{eq:ct-amd}\tag{AD}
    \begin{gathered}
    \dot{\vx}_t = \frac{\dot{\alpha}_t(\nabla \psi^*(\vz_t)-\vx_t)}{\alpha_t},\\
    \dot{\vz}_t = - \dot{\alpha}_t \nabla f(\vx_t).
    \end{gathered}
\end{equation} 
\markupadd{Using Eq.~\eqref{eq:hamilton-eq},} this dynamics can be shown to correspond to the following  Hamiltonian:~\footnote{This Hamiltonian was obtained in the discussions between J.~Diakonikolas and Lorenzo Orecchia.}
\begin{equation}\label{eq:amd-ham}
    \ch(\overline{\vx}, \vz, \tau) = \tau f({\overline{\vx}}/{\tau}) + \psi^*(\vz),\footnote{\markupadd{To avoid division by zero, the dynamics can be started from time $\tau = 1.$ Alternatively, one can replace $\tau$ by $\tau+1$ and start the dynamics from $\tau = 0$. In the sequel, we will always assume that $\alpha_0$ is bounded away from zero.}}
\end{equation}
after a suitable time reparametrization $\tau = \alpha_t$ (see also, e.g.,~\cite{wibisono2016variational}), 
where $\alpha_t$ is a strictly increasing function of time $t$ and $\overline{\vx} = \tau \vx$. To see that the dynamics from~\eqref{eq:ct-amd} corresponds to the equations of motion of the Hamiltonian~\eqref{eq:amd-ham}, observe that 
$$
\frac{\dd }{\dd t}\overline{\vx}_t = \frac{\dd \tau}{\dd t} \frac{\dd }{\dd \tau}\overline{\vx}_t = \dot{\alpha}_t \nabla_{\vz}\ch(\overline{\vx}_t, \vz_t, \tau) = \dot{\alpha}_t \nabla\psi^*(\vz_t),
$$
which, using $\overline{\vx}_t = \alpha_t \vx_t$, is exactly the first equation from~\eqref{eq:ct-amd}. Similarly, the second equation of motion for the Hamiltonian from Eq.~\eqref{eq:amd-ham} $\dot{\vz}_t = - \dot{\alpha_t}\nabla_{\vx}\ch(\overline{\vx}_t, \vz_t, \alpha_t) = - \dot{\alpha_t}\nabla f(\overline{\vx}_t/\alpha_t)$ is exactly the second equation from~\eqref{eq:ct-amd}, as $\overline{\vx}_t = \alpha_t \vx_t$. 

We now show that it is possible to generalize the Hamiltonian from Eq.~\eqref{eq:amd-ham} and its resulting equations of motion to capture a much broader class of convergent momentum-based methods that contains a generalization of Polyak's heavy ball method~\cite{polyak1964some}. In particular, consider:
\begin{equation}
\label{eq:momentum-ham}
    \chm(\overline{\vx}, \vz, \tau) = h(\tau)f({\overline{\vx}}/{\tau}) + \psi^*(\vz),
\end{equation}
where, as before, $\overline{\vx} = \tau\vx$, $h(\tau)$ is a positive function of $\tau,$ and we reparametrize time as $\tau = \alpha_t.$ We will mainly be considering the case $h(\tau) = {\tau}^{\lambda}$ for $\lambda \in [0, 2]$ ({see Sections~\ref{appx:conv-in-fun-val} } and \ref{sec:dt-methods}). 
The resulting equations of motion (after time reparametrization) of this Hamiltonian are:
\begin{equation}\label{eq:ct-mom-dyn}\tag{MoD}
    \begin{gathered}
        \dot{\vx}_t = \frac{\dot{\alpha}_t(\nabla\psi^*(\vz_t)-\vx_t)}{\alpha_t},\\
        \dot{\vz}_t = - h(\alpha_t)\frac{\dot{\alpha}_t}{\alpha_t}\nabla f(\vx_t).
    \end{gathered}
\end{equation}

\markupadd{Observe that the assumption that the vector space $E$ is non-Euclidean, i.e., that $\vx_t$ and $\vz_t$ belong to different (mutually dual) spaces, is seamlessly handled through the use of the map $\nabla \psi^*: E^* \to \cx \subseteq E.$  Due to Fact~\ref{fact:smoothness-sc-duality}, $\nabla \psi^*$ is a Lipschitz-continuous map whenever $\psi$ is strongly convex, which is a basic assumption we make throughout. Further, it is not hard to see that if $\vx_0 \in \cx,$ then $\vx_t \in \cx,$ $\forall t> 0.$ To see this, observe that, after suitably rearranging the terms, we can equivalently write the first equation from~\eqref{eq:ct-mom-dyn} as}
\begin{equation*}
    \markupadd{\frac{\dd (\alpha_t \vx_t)}{\dd t} = \dot{\alpha}_t \nabla \psi^*(\vz_t).}
\end{equation*}
\markupadd{Integrating both sides of the last equation, it follows that $$\vx_t = \frac{\alpha_0}{\alpha_t}\vx_0 + \frac{1}{\alpha_t}\int_0^t \nabla \psi^*(\vz_{\tau})\dot{\alpha}_\tau \dd \tau.$$ 
Thus, as $\vx_0 \in \cx$ and $\nabla \psi^*(\vz) \in \cx$, $\forall \vz \in E^*$ (see Fact~\ref{fact:danskin}), it follows that $\vx_t \in \cx$. Additionally, if the dynamics is started from the relative interior of the set $\cx,$ then $\vx_t$ always remains in the relative interior. As $\nabla \psi^*$ is a Lipschitz continuous map, unlike in the inertial approach, no non-elastic shocks can arise at the boundary. We note that the use of strongly convex functions $\psi$ resulting in smooth functions $\psi^*$ can also be viewed as constraint regularization. }

Clearly, the Hamiltonian $\chm$ and its equations of motion~\eqref{eq:ct-mom-dyn} generalize the accelerated dynamics: \eqref{eq:amd-ham} and~\eqref{eq:ct-amd} correspond to the case $h(\tau) = \tau.$ It is possible to show that the class of methods captured by the equations of motion of~\eqref{eq:momentum-ham} also contains a generalization of Polyak's heavy ball method, as shown in the following proposition. 
\begin{proposition}
Polyak's heavy ball method is equivalent to~\eqref{eq:ct-mom-dyn} when $\cx = \rr^n,$ $\|\cdot\| = \|\cdot\|_2,$ $\psi^*(\vz) = \frac{1}{2\mu}\|\vz\|_2^2,$ $h(\tau) = {\tau}^0 = 1,$ and $\frac{\dot{\alpha}_t}{\alpha_t} = \eta >0.$
\end{proposition}
\begin{proof}
Under the assumptions of the proposition, $\dot{\vx}_t = \eta(\frac{1}{\mu}\vz_t - \vx_t)$ and $\dot{\vz}_t = - \eta \nabla f(\vx_t).$ Hence, $\ddot{\vx}_t = - \eta \dot{\vx}_t - \frac{\eta^2}{\mu}\nabla f(\vx_t).$ For suitable choices of $\eta, \, \mu,$ this is equivalent to~\eqref{eq:HBD} from~\cite{polyak1964some}.
 \end{proof}
 
\markupadd{More generally, we can relate \eqref{eq:ct-mom-dyn} to the inertial dynamics from Eq.~\eqref{eq:inertial-ode} in the special case of \emph{unconstrained Euclidean setups} as follows. Let $\cx \equiv E$, $\|\cdot \| = \|\cdot \|_2$ and $\psi^*(\vz) = \frac{1}{2\mu}\|\vz\|^2$ for some $\mu > 0.$ Then $\nabla\psi^*(\vz) = \frac{1}{\mu}\vz,$ and we have:}
\begin{equation*}
    \markupadd{\frac{\dd (\alpha_t \vx)}{\dd t} = \dot{\alpha}_t \vx_t + \alpha_t \dot{\vx}_t = \dot{\alpha}_t  \vz_t/\mu.}
\end{equation*}
\markupadd{Dividing both sides by $\dot{\alpha}_t > 0$ and differentiating w.r.t.~$t$, we have, using the second equation in \eqref{eq:ct-mom-dyn}:}
\begin{equation*}
    \markupadd{\Big(1+ \frac{\dot{\alpha_t}^2 - \alpha_t\ddot{\alpha}_t}{{\dot{\alpha_t}}^2}\Big)\dot{\vx}_t + \frac{\alpha_t}{\dot{\alpha}_t}\ddot{\vx}_t = - \frac{1}{\mu}h(\alpha_t) \frac{\dot{\alpha}_t}{\alpha_t}\nabla f(\vx_t).}
\end{equation*}
\markupadd{Rearranging the last inequality:}
\begin{equation}\label{eq:inert-ode-spec-case}
   \markupadd{\ddot{\vx}_t + \frac{\dot{\alpha}_t}{\alpha_t}\Big(1+ \frac{\dot{\alpha_t}^2 - \alpha_t\ddot{\alpha}_t}{{\dot{\alpha_t}}^2}\Big)\dot{\vx}_t + \frac{1}{\mu}h(\alpha_t) \Big(\frac{\dot{\alpha}_t}{\alpha_t}\Big)^2\nabla f(\vx_t) = 0,}
\end{equation}
 \markupadd{which is precisely the inertial ODE from Eq.~\eqref{eq:inertial-ode} with $\alpha_1(t) = \frac{\dot{\alpha}_t}{\alpha_t}\big(1+ \frac{\dot{\alpha_t} - \ddot{\alpha}_t}{{\dot{\alpha_t}}^2}\big)$ and $\alpha_2(t) = \frac{1}{\mu}h(\alpha_t) \big(\frac{\dot{\alpha}_t}{\alpha_t}\big)^2.$ In particular, for $\alpha_t = t^p,$ $p > 0,$ $\mu = p$, and $h(\alpha_t) = {\alpha_t}^{2/p}/p,$ Eq.~\eqref{eq:inert-ode-spec-case} reduces to}
 \begin{equation*}
     \markupadd{\ddot{\vx}_t + \frac{p+1}{t}\dot{\vx}_t + \nabla f(\vx_t) = 0,}
 \end{equation*}
 \markupadd{which is the damped inertial ODE studied by Su et al.~\cite{SuBC16}.}
 
 The main usefulness of Hamiltonian~$\chm$ is that it can be used to argue about convergence in both the function value (for convex optimization problems) and convergence to stationary points (for potentially nonconvex problems). In the following lemma we exhibit two different conserved quantities (or invariants) of~\eqref{eq:momentum-ham} that can be used towards this goal.
 \begin{restatable}{lemma}{ctcq}\label{lemma:ct-gen-mom-cq}
 Let $\vx_t,\, \vz_t$ evolve according to~\eqref{eq:ct-mom-dyn} for an arbitrary initial point $\vx_0 = \nabla \psi^*(\vz_0) \in \cx$ and some differentiable $\psi^*(\cdot).$ Denote $\beta_t = h(\alpha_t)\alpha_t$. Then, $\forall t \geq 0,$ $\frac{\dd}{\dd t}\cc_t^f = 0$ and $\frac{\dd}{\dd t}\cc_t = 0$, where:
 \begin{align}
    \cc_t^f \defeq & h(\alpha_t) f(\vx_t) - \int_0^t f(\vx_{\tau})\frac{\dd (h(\alpha_{\tau}))}{\dd \tau}\dd \tau  + \int_0^t h(\alpha_t)\frac{\dot{\alpha}_{\tau}}{\alpha_{\tau}}\innp{\nabla f(\vx_{\tau}), \vx_{\tau}}\dd\tau + \psi^*(\vz_t), \label{eq:ct-fun-cq}\\
\cc_t \defeq &\, \beta_t f(\vx_t)-{\beta_0}f(\vx_{0}) - \int_0^t \dot{\beta}_\tau{\dd\tau} f(\vx_{\tau})\dd \tau \label{eq:ct-cq} + \alpha_0 D_{\psi^*}(\vz_t, \vz_{0})+ \int_0^t D_{\psi^*}(\vz_t, \vz_{\sigma})\dot{\alpha}_\sigma\dd\sigma.
\end{align}
 \end{restatable}
 The proof of Lemma~\ref{lemma:ct-gen-mom-cq} is provided in Appendix~\ref{appx:ct-proofs}. 
 
 Let us now provide some context for how conserved quantities $\cc_t^f$ and $\cc_t$ lead to convergence in function value and convergence in gradient norm, respectively. 
 First, when $h(\alpha_t) = \alpha_t$ (in which case~\eqref{eq:ct-mom-dyn} is equivalent to~\eqref{eq:ct-amd}), conservation of $\cc_t^f$ can be shown to be equivalent to the conservation of the scaled approximate duality gap from~\cite{thegaptechnique,AXGD}. More generally, as we show in Section~\ref{appx:conv-in-fun-val}, the conservation of $\cc_t^f$ can be used to upper bound  the optimality gap $f(\vxh_t) - f(\vx^*)$ for some $\vxh_t \in \cx$ that is constructed as a convex combination of $\{\vx_{\tau}\}_{\tau \in [0, t]}.$ 
 
 Consider now $\cc_t.$ If $\frac{\dd}{\dd t}\cc_t = 0$, then it is not hard to check that it must be the case that $\cc_t = 0,$ $\forall t.$ Equivalently, as it also holds that $\frac{\cc_t}{ h(\alpha_t){\alpha_t}} = 0,$ $\forall t$, we have:
 \begin{equation}\label{eq:implications-of-cct}
     \begin{aligned}
  & f(\vx_t)-\frac{h(\alpha_0){\alpha_0}f(\vx_{0}) + \int_0^t \frac{\dd( h(\alpha_\tau) \alpha_\tau)}{\dd\tau} f(\vx_{\tau})\dd \tau} {h(\alpha_t){\alpha_t}} \\
&\hspace{1in}= -\frac{\alpha_0 D_{\psi^*}(\vz_t, \vz_{0}) + \int_0^t D_{\psi^*}(\vz_t, \vz_{\sigma})\dot{\alpha}_\sigma\dd\sigma}{h(\alpha_t){\alpha_t}}.
\end{aligned}
 \end{equation}
 Observe that the right-hand side of \eqref{eq:implications-of-cct} is always non-positive, as $\psi^*$ is assumed to be convex. Suppose for now that the right-hand side of \eqref{eq:implications-of-cct} is strictly negative and less than $-\delta$ for some $\delta >0$. We then have: 
 $
 f(\vx_t) - \frac{h(\alpha_0){\alpha_0}}{h(\alpha_t)\alpha_t}f(\vx_{0}) - \frac{1}{h(\alpha_t)\alpha_t}\int_0^t \frac{\dd( h(\alpha_\tau) \alpha_\tau)}{\dd\tau} f(\vx_{\tau})\dd \tau < -\delta.
 $
 In other words, the function value at the last point $\vx_t$ is strictly smaller than a weighted average of function values at points $\vx_{\tau}$ for $\tau \in [0, t].$ This means that the (weighted) average function value must be strictly decreasing with time. As the function is bounded below, after some finite time it must be that the right-hand side is at least $-\delta.$ Observe that this argument can be made for any $\delta > 0.$ The main idea in the analysis is to show that the inequality
 \begin{equation} \label{eq:bregman-stat-cond}
    \frac{\alpha_0 D_{\psi^*}(\vz_t, \vz_{0}) + \int_0^t D_{\psi^*}(\vz_t, \vz_{\sigma})\dot{\alpha}_\sigma\dd\sigma}{h(\alpha_t){\alpha_t}} \leq \delta
 \end{equation}
 implies that the dynamics must visit at least one point $\vx$ such that $\|\nabla f(\vx)\|_* \leq \epsilon(\delta).$ 
\section{Convergence in Function Value}\label{appx:conv-in-fun-val}

In this section, we show that the invariants implied by the Hamiltonian that generates the momentum-based methods can be used to argue about convergence in function value. We start by arguing about the continuous-time case, and then show how the same invariant can be used analogously to argue about convergence of discretized versions of~\eqref{eq:ct-mom-dyn}. 

All the results will be obtained for the following choice of $h(\alpha_t):$
$$
h(\alpha_t) = {\alpha_t}^{\lambda}, \quad \text{where } \lambda \in [0, 1], 
$$
with the same relationship holding between their corresponding discrete-time counterparts ($A_k$ and $H_k$). This choice of $h(\alpha_t)$ interpolates between the accelerated method~\eqref{eq:ct-amd} (for $\lambda = 1$) and the generalized heavy ball method (for $\lambda = 0$).

%
%
\subsection{Convergence of the Continuous-Time Dynamics}
%
%
%
We now show how Lemma~\ref{lemma:ct-gen-mom-cq} can be used to argue about the convergence in function value of~\eqref{eq:ct-mom-dyn}.

\begin{lemma}\label{lemma:ct-fun-val-conv}
Let $\vx_t,\, \vz_t$ evolve according to~\eqref{eq:ct-mom-dyn} for $h(\alpha_t) = {\alpha_t}^{\lambda},$ $\lambda \in [0, 1]$ and $\vx_0 = \nabla \psi^*(\vz_0) \in \mathrm{rel}\, \mathrm{int}(\cx).$ If $\lambda = 0,$ assume that $\frac{\dot{\alpha}_t}{\alpha_t} = \eta > 0.$ Denote: 
$$
\vxh_t = \begin{cases}
\lambda\frac{{\alpha_t}^{\lambda} \vx_t \,+\, (1-\lambda)\int_0^t \dot{\alpha}_{\tau}{\alpha_{\tau}}^{1-\lambda}\vx_{\tau}\dd\tau \,+\, \frac{1-\lambda}{\lambda}\vx_0}{{\alpha_t}^{\lambda}}, &\text{ if } \lambda \in (0, 1],\\
\frac{\vx_t + \eta\int_0^t \vx_{\tau}\dd\tau}{1+ \eta t}, & \text{ if } \lambda = 0.
\end{cases}
$$ 
Then, $\forall t \geq 0:$
$$
f(\vxh_t) - f(\vx^*)\leq \begin{cases}
\lambda\frac{\frac{\alpha_0}{\lambda}(f(\vx_0) - f(\vx^*)) + D_{\psi}(\vx^*, \vx_0)}{{\alpha_t}^{\lambda}},  &\text{ if } \lambda \in (0, 1],\\
\frac{f(\vx_0) - f(\vx^*) + D_{\psi}(\vx^*, \vx_0)}{1 + \eta t}, & \text{ if } \lambda = 0.
\end{cases}
$$
\end{lemma}
\begin{proof}
Lemma~\ref{lemma:ct-gen-mom-cq} implies that $\cc_t^f = \cc_0^f,$ $\forall t \geq 0.$ Hence, as $h(\alpha_t) ={\alpha_t}^{\lambda}:$
\begin{equation}\label{eq:fun-val-1}
\begin{aligned}
    {\alpha_t}^{\lambda} f(\vx_t) - {\alpha_0}^{\lambda}f(\vx_0)& - \lambda \int_0^t \dot{\alpha}_{\tau}{\alpha_{\tau}}^{\lambda -1} f(\vx_{\tau})\dd\tau\\ 
    &= \psi^*(\vz_0) - \psi^*(\vz_t) - \int_0^t \dot{\alpha}_{\tau}{\alpha_{\tau}}^{\lambda-1}\innp{\nabla f(\vx_{\tau}), \vx_{\tau}}\dd\tau.
\end{aligned}
\end{equation}
Write $- \int_0^t \dot{\alpha}_{\tau}{\alpha_{\tau}}^{\lambda-1}\innp{\nabla f(\vx_{\tau}), \vx_{\tau}}\dd\tau$ as:
\begin{equation}\label{eq:fun-val-2}
\begin{aligned}
    &- \int_0^t \dot{\alpha}_{\tau}{\alpha_{\tau}}^{\lambda-1} \innp{\nabla f(\vx_{\tau}), \vx_{\tau}}\dd\tau\\
    &\hspace{.5in}= \int_0^t \dot{\alpha}_{\tau}{\alpha_{\tau}}^{\lambda-1}\innp{\nabla f(\vx_{\tau}), \vx^* - \vx_{\tau}}\dd\tau - \int_0^t \dot{\alpha}_{\tau}{\alpha_{\tau}}^{\lambda-1}\innp{\nabla f(\vx_{\tau}), \vx^*}\dd\tau.
    \end{aligned}
\end{equation}
Observe that, by convexity of $f:$
\begin{equation}\label{eq:fun-val-3}
    \int_0^t \dot{\alpha}_{\tau}{\alpha_{\tau}}^{\lambda-1}\innp{\nabla f(\vx_{\tau}), \vx^* - \vx_{\tau}}\dd\tau \leq \int_0^t \dot{\alpha}_{\tau}{\alpha_{\tau}}^{\lambda-1} (f(\vx^*) - f(\vx_{\tau}))\dd\tau.
\end{equation}
By the definition of $\vz_t$ from~\eqref{eq:ct-mom-dyn},
\begin{equation}\label{eq:fun-val-4}
    \begin{aligned}
    - \int_0^t \dot{\alpha}_{\tau}{\alpha_{\tau}}^{\lambda-1}\innp{\nabla f(\vx_{\tau}), \vx^*}\dd\tau = \int_0^t \innp{\dot{\vz}_{\tau}, \vx^*}= \innp{\vz_t - \vz_0, \vx^*}.
    \end{aligned}
\end{equation}
The next step is to combine $\innp{\vz_t - \vz_0, \vx^*}$ with $\psi^*(\vz_0) - \psi^*(\vz_t) $ to write them in the form of Bregman divergences. In particular, define $\vz^*$ so that $\nabla\psi^*(\vz^*) = \vx^*.$ Then:
\begin{align*}
    \psi^*(\vz_0) - \psi^*(\vz_t) + \innp{\vz_t - \vz_0, \vx^*} = D_{\psi^*}(\vz_0, \vz^*) - D_{\psi^*}(\vz_t, \vz^*) \leq D_{\psi^*}(\vz_0, \vz^*).
\end{align*}
Using Fact~\ref{fact:danskin} (which implies $\psi^*(\vz) = \innp{\nabla \vz, \nabla \psi^*(\vz)} - \psi(\nabla \psi^*(\vz))$) and $\vz_0 = \nabla \psi(\vx_0)$ (which follows from the assumption that $\vx_0$ is from the relative interior of $\cx$), it is not hard to show that: 
$
D_{\psi^*}(\vz_0, \vz^*) = D_{\psi}(\nabla \psi^*(\vz^*), \nabla \psi^*(\vz_0)) = D_{\psi}(\vx^*, \vx_0).
$ 
Combining with~\eqref{eq:fun-val-1}-\eqref{eq:fun-val-4}:
\begin{align*}
    {\alpha_t}^{\lambda} f(\vx_t) - {\alpha_0}^{\lambda}f(\vx_0)& - \lambda \int_0^t \dot{\alpha}_{\tau}{\alpha_{\tau}}^{\lambda -1} f(\vx_{\tau})\dd\tau\\ 
    &\leq \int_0^t \dot{\alpha}_{\tau}{\alpha_{\tau}}^{\lambda-1} (f(\vx^*) - f(\vx_{\tau}))\dd\tau + D_{\psi}(\vx^*, \vx_0). 
\end{align*}
Assume first that $\lambda > 0$. Integrating and rearranging the terms in the last inequality:
\begin{align*}
    {\alpha_t}^{\lambda} f(\vx_t)\, + \, &(1-\lambda)\int_{0}^t \dot{\alpha}_{\tau}{\alpha_{\tau}}^{\lambda -1} f(\vx_{\tau})\dd\tau + \frac{1-\lambda}{\lambda}f(\vx_0) - \frac{{\alpha_t}^{\lambda}}{\lambda}f(\vx^*)\\
    &\leq \frac{\alpha_0}{\lambda}(f(\vx_0) - f(\vx^*)) + D_{\psi}(\vx^*, \vx_0).
\end{align*}
It remains to divide both sides  by $\frac{{\alpha_t}^{\lambda}}{\lambda}$ and apply Jensen's inequality.

If $\lambda = 0,$ then, assuming $\frac{\dot{\alpha_t}}{\alpha_t} = \eta$:
\begin{align*}
    f(\vx_t) + \eta \int_0^t f(\vx_{\tau})\dd\tau - (1 + \eta t)f(\vx^*) \leq f(\vx_0) - f(\vx^*) + D_{\psi}(\vx^*, \vx_0).
\end{align*}
It remains to divide both sides by $1+\eta t$ and apply Jensen's inequality.
\end{proof}
Observe that when $\lambda = 1$ (that is, when \eqref{eq:ct-mom-dyn} is equivalent to \eqref{eq:ct-amd}), $\vxh_t = \vx_t,$ and we recover the standard guarantee on the last iterate of the accelerated dynamics~\cite{krichene2015accelerated,thegaptechnique,AXGD}. When $\lambda = 0,$ we obtain a $1/t$ convergence rate for the generalization of the heavy ball method. The result applies to constrained optimization and \markupdelete{Banach} \markupadd{non-Euclidean} spaces. We note that a generalization of the heavy ball method to constrained convex optimization was previously considered in~\cite{attouch2000heavy}. However, the result from~\cite{attouch2000heavy} applies only to Hilbert spaces and provides weak (asymptotic) convergence results. The second-order ODE considered in~\cite{attouch2000heavy} seems to correspond to a different continuous-time dynamics than~\eqref{eq:ct-mom-dyn} with $h(\alpha_t) = 1$, and it is unclear how to compare it to~\eqref{eq:ct-mom-dyn}.
\subsection{Discrete-Time Convergence}
Define the discrete-time counterpart to the continuous-time conserved quantity $\cc_t^f,$ $\cc_k^f,$ as:
$$
\cc_k^f = H_k f(\vy_k) - \sum_{i=1}^k h_i f(\vx_i) + \sum_{i=1}^k H_i \frac{a_i}{A_i} \innp{\nabla f(\vx_i), \vx_i} + \psi^*(\vz_k),
$$
where $A_k = \sum_{i=0}^k a_i$, $H_k = \sum_{i=0}^k h_i$, and $H_k = {A_k}^{\lambda}.$ 

The discretization of the continuous-time dynamics that we will use is:
\begin{equation}\label{eq:gen-mom-method-f}\tag{GMD\textsubscript{f}}
    \begin{gathered}
    \vx_k = \frac{H_{k-1}/H_k}{H_{k-1}/H_k + a_k/A_k}\vy_{k-1} + \frac{a_k/A_k}{H_{k-1}/H_k + a_k/A_k} \nabla \psi^*(\vz_{k-1}),\\
    \vz_k = \vz_{k-1} - H_k \frac{a_k}{A_k}\nabla f(\vx_k),\\
    \vy_k = \vx_k + \frac{a_k}{A_k}( \nabla \psi^*(\vz_{k}) - \nabla \psi^*(\vz_{k-1})).
    \end{gathered}
\end{equation}
The motivation for this particular choice of the discretization will become clear from Proposition~\ref{prop:fun-conserv-q}. (In particular, the discretization was chosen to ensure that $\cc_k^f \leq \cc_{k-1}^f.$) Note that when $H_k = A_k$ ($\lambda = 1$), the method is precisely the AGD+ method from~\cite{cohen2018acceleration}. \markupadd{Note that AGD+ is closely related to Nesterov's method~\cite{Nesterov1983} and the accelerated proximal method of G\"{u}ler~\cite{guler1992new}: both Nesterov's and G\"{u}ler's methods can be seen as alternative discretizations of \eqref{eq:ct-amd} (or~\eqref{eq:ct-mom-dyn} with $h(\tau) = \tau$), where $\vy_k$ is replaced by a correction step, which is the gradient descent step for Nesterov's method and the proximal step for G\"{u}ler's method. This interpretation of Nesterov's method can also be found in~\cite{thegaptechnique}.}

For~\eqref{eq:gen-mom-method-f} to apply to constrained minimization, we need to show that the iterates $\vy_k$ remain in the feasible set. This is established by the following proposition.

\begin{proposition}\label{prop:GMD-f-feasibility}
Let $\vx_k,\, \vy_k, \, \vz_k$ evolve as in~\eqref{eq:gen-mom-method-f}, where the initial point satisfies $\vy_0 = \nabla \psi^*(\vz_0) \in \mathrm{rel \, int}\cx$, $\mu \leq L,$ and $a_k,\, A_k,\, H_k >0$ satisfy: $A_k = \sum_{i=0}^k a_i$,  $\frac{{a_k}^2}{{A_k}^2} = c \frac{\mu}{L H_k}$, for some $c\in (0, 1]$, and $H_k = {A_k}^{\lambda}$ for $\lambda \in [0, 1]$. Then $\vy_k \in \cx,$ $\forall k \geq 0.$
\end{proposition}
\begin{proof}
The claim clearly holds for $k = 0,$ by the initialization. To simplify the notation, denote $\theta_k = \frac{a_k}{A_k}$, $\theta_k' = \frac{a_k/A_k}{H_{k-1}/H_k + a_k/A_k} = \frac{\theta_k}{H_{k-1}/H_k + \theta_k},$ $\vv_k = \nabla \psi^*(\vz_k).$ By the definition of a convex conjugate (Definition~\ref{def:cvx-conj}) and Fact~\ref{fact:danskin}, $\vv_k \in \cx,$ $\forall k.$



Under the assumptions of the proposition, it is not hard to see that $\theta_k \leq \theta_{k-1},$ $\forall k \geq 1.$ Namely, this condition is equivalent to $\frac{c\mu}{LH_{k}} \leq \frac{c\mu}{L H_{k-1}},$ which, by the definition of $H_k,$ is equivalent to ${A_{k-1}}^{\lambda} \leq {A_k}^{\lambda}.$ As $A_k$ is non-increasing with $k$ (as $A_k = \sum_{i=0}^k a_i$ and $a_i \geq 0,$ $\forall i$) and $\lambda \geq 0,$ we clearly have that ${A_{k-1}}^{\lambda} \leq {A_k}^{\lambda},$ and, thus $\theta_k \leq \theta_{k-1}.$

To prove the proposition, we will first show that $\vy_k$ can be expressed as a non-negative linear combination of $\{\vv_i\}_{i=0}^k.$ We subsequently show by induction on $k$ that the coefficients of that linear combination must sum to one, which completes the proof.

Using~\eqref{eq:gen-mom-method-f}, we can write $\vy_k$ in the following recursive form:
$$
\vy_k = (1-\theta_k')\vy_{k-1} + \theta_k' \vv_{k-1} + \theta_k(\vv_{k} - \vv_{k-1}).
$$
Applying this definition recursively over $i = 0,\dots, k$ and using $\vy_0 = \vv_0,$ we have
$$
\vy_k = \sum_{i=0}^k \gamma_{i, k} \vv_i,
$$
$$
\gamma_{i, k} = \begin{cases}
\theta_k, & \text{ if } i = k;\\
\Big[\prod_{j= i + 2}^k (1- \theta_j') \Big] \Big[\theta_{i+1}'(1 - \theta_i) + \theta_i - \theta_{i+1}\Big], &\text{ if } 1\leq i\leq k-1;\\
\Big[\prod_{j= 1}^k (1- \theta_j'), &\text{ if } i = 0,
\end{cases}
$$
where, by convention, we take $\prod_i^j (\cdot) = 1$ whenever $j < i.$ 
Given that for all $i \geq 0,$ we have $\theta_i, \theta_i' \in [0, 1]$ and $\theta_{i+1} \leq \theta_i,$ it immediately follows that $\gamma_{i, k} \geq 0,$ $\forall i \in \{0, 1, ... , k\}.$

We now show by induction on $k$ that it must be the case that $\sum_{i=0}^k\gamma_{i, k} = 1.$ This clearly holds for $k=0.$ Suppose that it holds for some $k - 1 \geq 0.$ Then $\vy_{k-1} \in \cx.$ As $\vx_k = (1-\theta_k')\vy_{k-1} + \theta_k' \vv_{k-1},$ it follows that $\vx_k \in \cx,$ and, moreover, $\vx_k$ is a convex combination of $\{\vv_i\}_{i=0}^{k-1}.$ Now, observe from the definition of $\vx_k$ that $\theta_k'\vv_{k-1}' = \vx_k - (1-\theta_k')\vy_{k-1}.$ Hence, we can express $\vy_k$ as:
$$
\vy_k = (1-\theta_k/\theta_k')\vx_k  + \frac{\theta_k}{\theta_k'}(1-\theta_k')\vy_{k-1} + \theta_k \vv_k.
$$
As $1 - \frac{\theta_k}{\theta_k'} + \frac{\theta_k}{\theta_k'}(1-\theta_k') + \theta_k = 1$ and each $\vx_k, \, \vy_{k-1},\, \vv_k$ are convex combinations of $\{\vv_i\}_{i=0}^k,$ it follows that $\sum_{i=0}^k \gamma_{i, k} = 1,$ which, together with the fact that $\gamma_{i, k} \geq 0,$ $\forall i,$ completes the proof.
\end{proof}

\begin{proposition}\label{prop:fun-conserv-q}
Let $\vx_k,\, \vy_k, \, \vz_k$ evolve according to~\eqref{eq:gen-mom-method-f}, where $\psi:\cx \rightarrow \mathbb{R}$ is a $\mu$-strongly convex function, $\vy_0 = \nabla \psi^*(\vz_0) \in \mathrm{rel \, int}\cx$ and $a_k,\, A_k,\, H_k >0$ satisfy: $A_k = \sum_{i=0}^k a_i$ and $\frac{{a_k}^2}{{A_k}^2}\leq \frac{\mu}{L H_k}$. Then $\cc_k^f \leq \cc_{k-1}^f,$ $\forall k \geq 1.$
\end{proposition}
\begin{proof}
By the definition of $\cc_k^f,$ we have:
\begin{equation}\label{eq:fcq-1}
\begin{aligned}
    \cc_k^f - \cc_{k-1}^f =&\ H_k f(\vy_k) - H_{k-1}f(\vy_{k-1}) - h_k f(\vx_k) + H_k \frac{a_k}{A_k}\innp{\nabla f(\vx_k), \vx_k}\\
    &+ \psi^*(\vz_k) - \psi^*(\vz_{k-1}).
\end{aligned}
\end{equation}
Observe first, by smoothness and convexity of $f:$
\begin{equation}\label{eq:fcq-2}
    \begin{aligned}
    H_k f(\vy_k) - H_{k-1} & f(\vy_{k-1}) - h_k f(\vx_k)\\
    &= H_k(f(\vy_k) - f(\vx_k)) + H_{k-1}(f(\vx_k) - f(\vy_{k-1}))\\
    &\leq \innp{\nabla f(\vx_k), H_k\vy_k - H_{k-1}\vy_{k-1} - h_k \vx_k} + H_k\frac{L}{2}\|\vy_k - \vx_k\|^2.
    \end{aligned}
\end{equation}
On the other hand, by the definitions of a Bregman divergence and $\vz_k:$
\begin{equation}\label{eq:fcq-3}
    \begin{aligned}
    \psi^*(\vz_k) - \psi^*(\vz_{k-1}) &= -D_{\psi^*}(\vz_{k-1}, \vz_k) - \innp{\nabla \psi^*(\vz_{k-1}), \vz_{k-1} - \vz_k}\\
    &= -D_{\psi^*}(\vz_{k-1}, \vz_k) - H_k \frac{a_k}{A_k}\innp{\nabla f(\vx_k), \nabla \psi^*(\vz_{k-1})}.
    \end{aligned}
\end{equation}
As $\psi$ is $\mu$-strongly convex, we have (by Fact~\ref{fact:bregman-properties}): $D_{\psi^*}(\vz_{k-1}, \vz_k) \geq \frac{\mu}{2}\|\nabla \psi^*(\vz_k) - \nabla \psi^*(\vz_{k-1})\|^2.$ Hence, combining~\eqref{eq:fcq-1}-\eqref{eq:fcq-3}:
\begin{align*}
    \cc_k^f - \cc_{k-1}^f \leq &\ H_k\frac{L}{2}\|\vy_k - \vx_k\|^2 - \frac{\mu}{2}\|\nabla \psi^*(\vz_k) - \nabla \psi^*(\vz_{k-1})\|^2\\
    &+ \Big\langle{\nabla f(\vx_k), H_k \vy_k - H_{k-1}\vy_{k-1} - h_k \vx_k + H_k \frac{a_k}{A_k}(\vx_k - \nabla \psi^*(\vz_k))}\Big\rangle.
\end{align*}
Note that we want to make the right-hand side of the last inequality non-positive. To do so, we can make the last term equal to zero by setting: $H_k \vy_k - H_{k-1}\vy_{k-1} - h_k \vx_k + H_k \frac{a_k}{A_k}(\vx_k - \nabla \psi^*(\vz_k)) = 0.$ To make the first two terms non-positive, we require $\vy_k - \vx_k = \frac{a_k}{A_k}(\nabla \psi^*(\vz_k) - \nabla\psi^*(\vz_{k-1})).$\footnote{Note that the multiplier $\frac{a_k}{A_k}$ on the right-hand side is necessary here for $\vx_k$ to be explicitly defined. Any other factor would make $\vx_k$ depend on $\nabla \psi^*(\vz_k)$, which is a function of $\vx_k$ (as $\vz_k = \vz_{k-1} - H_k\frac{a_k}{A_k}\nabla f(\vx_k)$), and would thus make $\vx_k$ be only implicitly defined.} Solving these last two equations for $\vx_k$ gives~\eqref{eq:gen-mom-method-f}. We conclude by using $H_k\frac{{a_k}^2}{{A_k}^2} \leq \frac{\mu}{L}.$
\end{proof}
To obtain a convergence rate for~\eqref{eq:gen-mom-method-f}, it remains to show that $\cc_k^f \leq \cc_0^f$ implies a convergence in function value for~\eqref{eq:gen-mom-method-f}, as was done for the continuous-time case in Lemma~\ref{lemma:ct-fun-val-conv}. 
\begin{theorem}\label{thm:conv-fun-val-d}
Let $\vx_k, \, \vy_k,\, \vz_k$ evolve according to~\eqref{eq:gen-mom-method-f}, where $\vx_0 = \vy_0 = \nabla \psi^*(\vz_0) \in \mathrm{rel \, int}\cx,$ $\psi: \cx \rightarrow \mathbb{R}$ is $\mu$-strongly convex, $H_k = {A_k}^{\lambda},$ $\lambda \in [0, 1],$ $A_k = \sum_{i=0}^k a_k,$ $a_0 = 1$, and $\frac{{a_k}^2}{{A_k}^2} = c\frac{\mu}{L H_i},$ for $c \in (0, 1],$ $k \geq 1.$ Define:
$$
\vxh_k = \frac{H_k \vy_k + \sum_{i=1}^k (\frac{a_i}{A_i}H_i - h_i)\vx_i}{H_0 + \sum_{i=1}^k \frac{a_i}{A_i}H_i}.
$$
Then, $\forall k \geq 1,$ $\vxh_k \in \cx$ and:
$$
f(\vxh_k) - f(\vx^*) \leq \begin{cases}
\frac{f(\vx_0) - f(\vx*) + D_{\psi}(\vx^*, \vx_0)}{1 + \sqrt{c\mu/L}k}, & \text{ if } \lambda = 0, k \geq 1\\
\Theta(1) \lambda^{\lambda} \frac{L}{c\mu}\frac{f(\vx_0) - f(\vx^*) + D_{\psi}(\vx^*, \vx_0)}{k^2}, & \text{ if } \lambda \in (0, 1],\, k = \Omega(1/\lambda).
\end{cases}
$$
\end{theorem}
\begin{proof}
By Proposition~\ref{prop:GMD-f-feasibility},  $\vy_k \in \cx,$ $\forall k.$ As $\vx_k \in \cx$ (as a convex combination of $\vy_{k-1}, \, \nabla \psi^*(\vz_{k-1})\in \cx$), we have that $\vxh_k$ is a convex combination of points from the feasible space $\cx,$ and, thus, it must be the case that $\vxh_k \in \cx,$ $\forall k.$

By Proposition~\ref{prop:fun-conserv-q}, $\cc_k^f \leq \cc_{0}^f.$ Hence:
\begin{equation}\label{eq:gmd-f-bnd-1}
    \begin{aligned}
    H_kf(\vy_k) - H_0 f(\vy_0) - \sum_{i=1}^k h_i f(\vx_i) \leq \psi^*(\vz_0) - \psi^*(\vz_k) - \sum_{i=1}^k H_i \frac{a_i}{A_i}\innp{\nabla f(\vx_i), \vx_i}.
    \end{aligned}
\end{equation}
As in Lemma~\ref{lemma:ct-fun-val-conv}, write $- \sum_{i=1}^k H_i \frac{a_i}{A_i}\innp{\nabla f(\vx_i), \vx_i}$ as:
\begin{equation}\label{eq:gmd-f-bnd-2}
    - \sum_{i=1}^k H_i \frac{a_i}{A_i}\innp{\nabla f(\vx_i), \vx_i} = \sum_{i=1}^k H_i \frac{a_i}{A_i}\innp{\nabla f(\vx_i), \vx^* - \vx_i} - \sum_{i=1}^k H_i \frac{a_i}{A_i}\innp{\nabla f(\vx_i), \vx^*}.
\end{equation}
By convexity of $f:$
\begin{equation}\label{eq:gmd-f-bnd-3}
    \sum_{i=1}^k H_i \frac{a_i}{A_i}\innp{\nabla f(\vx_i), \vx^* - \vx_i} \leq \sum_{i=1}^k H_i \frac{a_i}{A_i}(f(\vx^*) - f(\vx_i)).
\end{equation}
Let $\vz^*$ be such that $\nabla \psi^*(\vz^*) = \vx^*.$ Then, by the same arguments as in the proof of Lemma~\ref{lemma:ct-fun-val-conv}:
\begin{equation}\label{eq:gmd-f-bnd-4}
\begin{aligned}
    \psi^*(\vz_0) - \psi^*(\vz_k) - \sum_{i=1}^k H_i \frac{a_i}{A_i}\innp{\nabla f(\vx_i), \vx^*} &= D_{\psi^*}(\vz_0, \vz^*) - D_{\psi^*}(\vz_t, \vz^*)\\
    &\leq D_{\psi}(\vx^*, \vx_0). 
\end{aligned}
\end{equation}
Combining~\eqref{eq:gmd-f-bnd-1}-\eqref{eq:gmd-f-bnd-4}:
\begin{equation}\notag
    H_kf(\vy_k) - H_0 f(\vy_0) - \sum_{i=1}^k h_i f(\vx_i) \leq \sum_{i=1}^k H_i \frac{a_i}{A_i}(f(\vx^*) - f(\vx_i)) + D_{\psi}(\vx^*, \vx_0). 
\end{equation}
To complete the proof, it remains to rearrange the terms in the last equation. Notice that $\sum_{i=1}^k h_i = H_k - H_0,$ and, thus, the coefficients multiplying $f(\cdot)$ sum up to zero. Notice also that, as $H_i = {A_i}^{\lambda},$ $a_i = A_{i} - A_{i-1},$ and $h_i = H_i - H_{i-1},$ it must be the case that $H_i \frac{a_i}{A_i} - h_i \geq 0.$ We have:
\begin{align*}
H_k f(\vy_k) + \sum_{i=1}^k \Big(H_i \frac{a_i}{A_i} - h_i\Big)f(\vx_i) - &\Big(H_0 + \sum_{i=1}^k H_i \frac{a_i}{A_i}\Big)f(\vx^*)\\
&\leq H_0(f(\vx_0) - f(\vx^*)) + D_\psi(\vx^*, \vx_0).
\end{align*}
Dividing both sides of the last equation by $\Big(H_0 + \sum_{i=1}^k H_i \frac{a_i}{A_i}\Big)$ and applying Jensen's inequality:
$$
f(\vxh_t) - f(\vx^*) \leq \frac{H_0(f(\vy_0) - f(\vx^*)) + D_{\psi}(\vx^*, \vx_0)}{\Big(H_0 + \sum_{i=1}^k H_i \frac{a_i}{A_i}\Big)}.
$$
Recall that, as $a_0 = 1,$ we must have $H_0 = 1.$ To bound $H_0 + \sum_{i=1}^k H_i \frac{a_i}{A_i},$ we need to argue about the growth of $\frac{a_i}{A_i}H_i = \frac{a_i}{{A_i}^{1-\lambda}}.$ 
When $\lambda = 0,$ we have $\frac{a_i}{{A_i}^{1-\lambda}} = \sqrt{\frac{c\mu}{L}}.$ 
Assume now that $\lambda > 0$ and $k = \Omega(\frac{1}{\lambda})$. By assumption of the theorem, $\frac{{a_i}^2}{{A_i}^{2-\lambda}} = \frac{c\mu}{L}.$ If $a_i \propto \big(\frac{c\mu \lambda}{2L}\big)^{1/\lambda} i^{2/\lambda - 1}$ and $i = \Omega(\frac{1}{\lambda}),$ then $A_i \propto \frac{\lambda}{2} \big(\frac{c\mu \lambda}{2L}\big)^{1/\lambda} i^{2/\lambda},$ and it can be ensured that  $\frac{{a_i}^2}{{A_i}^{2-\lambda}} = \frac{c\mu}{L}$ holds. Thus, for $i \geq k/2,$  $\frac{a_i}{{A_i}^{1-\lambda}} = \Theta(1)\lambda^{\lambda}\frac{c\mu}{L}i.$ 
Hence, $\sum_{i=1}^k \frac{a_i}{A_i}H_i$ scales as $\Theta(k)\sqrt{\frac{c\mu}{L}}$ for $\lambda = 0$ and  $\Theta(k^2)\lambda^{\lambda}\frac{c\mu}{L}$ for $\lambda > 0$ and $k = \Omega(\frac{1}{\lambda}).$ 
\end{proof}
Observe that a generalization of the heavy ball method (obtained from~\eqref{eq:gen-mom-method-f} for $H_k = {A_k}^0 = 1$) converges at rate $1/k$ for smooth convex functions. The result applies to constrained minimization and general normed vector spaces. Note that such a (nonasymptotic) result was previously known only for the setting of unconstrained minimization in Euclidean spaces~\cite{ghadimi2015global}. 
%

%
%
%
%
\section{Convergence in Norm of the Gradient}\label{sec:dt-methods}
In this section we turn to convergence in terms of the norm of the gradient.  We focus on discrete-time methods, based on a discrete-time counterpart to the conserved quantity $\cc_t$. Throughout the section, we assume that $\cx \equiv E.$ We start the section with an overview of the approach and a structural (algorithm-independent) lemma that is used later in proving the convergence results. We then present the results for \markupdelete{Hilbert} \markupadd{Euclidean} spaces in Section~\ref{sec:hilbert}, and show how these results can be generalized to \markupdelete{Banach} \markupadd{non-Euclidean} spaces in Section~\ref{sec:banach}. 
%
%
%
\subsection{An Overview of the Approach and a Structural Lemma}

We consider a counterpart to  $\cc_t$ that was derived for general momentum methods and defined in Lemma~\ref{lemma:ct-gen-mom-cq}. This counterpart is defined as:
\begin{equation}\label{eq:discr-cq}
    \cc_k = {B_k} f(\vy_{k}) - \sum_{i=0}^k b_i f(\vy_{i}) + \sum_{i=0}^k a_i D_{\psi^*}(\vz_{k}, \vz_{i}),
\end{equation}
where $A_k = \sum_{i=0}^k a_i$, $B_k = \sum_{i=0}^k b_i,$ and $a_i, b_i > 0,$ $\forall i \geq 0.$ Going from continuous to the discrete time, $\alpha_t,\, h(\alpha_t),\, \alpha_t h(\alpha_t)$ translate into $A_k,\, H_k,\, B_k,$ respectively.

To characterize convergence to stationary points, we denote by $E_{k} \defeq \cc_{k} - \cc_{k-1}$ the discretization error between the iterations $k-1$ and $k$ (recall that in the continuous time domain, $\cc_t$ was conserved). As $\cc_0 = 0,$ we clearly have $\cc_k = \sum_{i=1}^{k} E_i.$ Given specific assumptions about the objective function (e.g., if it is convex or nonconvex, its degree of smoothness, etc.), we will need to argue that the total discretization error $\sum_{i=1}^{k} E_i$ is ``sufficiently small'' (possibly zero, or even negative) under those assumptions. In general, the magnitude of the discretization error will be determined by the step sizes $a_i, b_i,$ which will be one of the constraining factors determining the rate of convergence.
\paragraph{Average Function Value Decrease}
The following (algorithm-independent) lemma implies that if $\cc_k$ is non-increasing with $k$ (namely, if $E_k \leq 0,$ $\forall k$) then the average function value taken at all points constructed by the algorithm is decreasing with the iteration count $k.$ The lemma will be crucial in obtaining the results for the convergence in norm of the gradient. 
\begin{restatable}{lemma}{structural}\label{lemma:avg-fn-value-dec}
Let $\cc_k = {B_k} f(\vy_{k}) - \sum_{i=0}^k b_i f(\vy_{i}) + \sum_{i=0}^k a_i D_{\psi^*}(\vz_{k}, \vz_{i})$ and $E_k = \cc_{k} - \cc_{k-1},$ where $\forall i,$ $a_i, b_i >0$ and $\forall k,$ $A_k = \sum_{i=0}^k a_i,$ $B_k = \sum_{i=0}^k b_i$. 
Then:
\begin{align*}
\frac{1}{B_k}\sum_{i=0}^k b_i f(\vy_{i})\, = \; & f(\vy_{0}) - \sum_{i=1}^{k} \Big(\frac{1}{B_{i-1}} - \frac{1}{B_i}\Big) \sum_{j=0}^{i-1} a_j D_{\psi^*}(\vz_{i}, \vz_{j})\\
&+ \sum_{i=1}^{k}\Big(\frac{1}{B_{i-1}} - \frac{1}{B_k}\Big)E_i. 
\end{align*}
\end{restatable}
\begin{proof}
Denote $S_k = \sum_{i=0}^k b_i f(\vy_{i}).$ 
The proof is by induction on $k$. The base case is immediate, as, by the definition of $S_k$ and $B_k,$ we have $S_0 = b_0 f(\vy_{0}) = B_0 f(\vy_{0}).$  Now assume that the statement is true for $k \geq 0.$ Then, by the definition of $S_k,$ 
\begin{equation}\label{eq:S_k-change}
    S_{k+1} = S_k + b_{k+1} f(\vy_{k+1}).
\end{equation}
On the other hand, by the definitions of $\cc_k$ and $E_k,$ we have $\cc_k = \sum_{i=1}^{k}E_i,$ and, thus:
\begin{equation}\label{eq:fy-as-a-fn-of-S}
    f(\vy_{k+1}) = \frac{1}{B_{k}}\Big(S_k - \sum_{j=0}^{k}a_j D_{\psi^*}(\vz_{k+1}, \vz_{j}) + \sum_{i=1}^{k+1} E_i\Big).
\end{equation}
Combining Equations~\eqref{eq:S_k-change} and~\eqref{eq:fy-as-a-fn-of-S}:
\begin{align*}
    S_{k+1} = &\ \Big(1 + \frac{b_{k+1}}{B_k}\Big)S_k - \frac{b_{k+1}}{B_k} \sum_{j=0}^{k}a_j D_{\psi^*}(\vz_{k+1}, \vz_{j}) + \frac{b_{k+1}}{B_k}\sum_{i=1}^{k+1} E_i\\
    = &\ \frac{B_{k+1}}{B_k} S_k - B_{k+1}\sum_{j=0}^{k}\Big(\frac{1}{B_{k}} - \frac{1}{B_{k+1}}\Big)a_j D_{\psi^*}(\vz_{k+1}, \vz_{j})\\
    \ \ &+ B_{k+1}\sum_{i=1}^{k+1}\Big(\frac{1}{B_{k}} - \frac{1}{B_{k+1}}\Big)E_i,
\end{align*}
where we have used $B_{k+1} = B_k + b_{k+1}.$ Applying the inductive hypothesis and grouping terms into appropriate summations completes the proof.
\end{proof}
%
%
%
%
%
\subsection{Convergence to Stationary Points in \markupadd{Euclidean} Spaces}\label{sec:hilbert}
In this section, we take $(E, \|\cdot\|)$ to be a \markupdelete{Hilbert} \markupadd{Euclidean} space. 
The following simple claim is useful for passing from Bregman divergences to gradient norms.
\begin{claim}\label{claim:triangle-ineq}
Let $a, b > 0$. Then $a\|\vz + \Delta \vz\|^2 + b\|\vz\|^2 \geq \frac{ab}{a+b}\|\Delta \vz\|^2.$
\end{claim} 
\begin{proof}
As $\|\cdot\| = \sqrt{\innp{\cdot, \cdot}},$ we have that $a\|\vz + \Delta \vz\|^2 + b\|\vz\|^2 = (a+b)\|\vz\|^2 + 2a\innp{\vz, \Delta \vz} + a\|\Delta\vz\|^2.$ By the Cauchy-Schwarz inequality, $\innp{\vz, \Delta \vz} \leq \|\vz\|\|\Delta \vz\|.$ Since the claim holds trivially for $\|\Delta\vz\| = 0,$ assume $\|\Delta\vz\| \neq 0$ and let $c = \frac{\|\vz\|}{\|\Delta \vz\|}.$ Then $a\|\vz + \Delta \vz\|^2 + b\|\vz\|^2 \geq \|\Delta \vz\|^2((a+b)c^2 -2ac +a)$. As $(a+b)c^2 -2ac + a$ is minimized for $c = \frac{a}{a+b},$ the claim follows.
\end{proof}

To relate the Bregman divergences in the definition of $\cc_k$ to norms of the gradients, we will use the following application of Claim~\ref{claim:triangle-ineq}.
\begin{proposition}\label{prop:BD-to-grad-norms}
Let $\psi(\vx) = \frac{\mu}{2}\|\vx\|^2$ (so that $D_{\psi^*}(\vw, \vz) = \frac{1}{2\mu}\|\vw - \vz\|^2$). Then:
$$
\sum_{j=0}^{i-1} a_j D_{\psi^*}(\vz_{i}, \vz_{j}) \geq \frac{1}{2\mu}\sum_{j=1}^i \nu_j^i \Big\| \nabla f(\vx_{j}) \Big\|^2,
$$
where $\nu_i^i = \frac{{a_i}^2}{{A_i}^2}{H_i}^2\big(\frac{a_i}{2} + \frac{a_i a_{i-1}}{a_i + a_{i-1}}\big)$ and $\nu_j^i = \frac{a_{j-1}{a_j}^3}{a_{j-1} + a_j}\frac{{H_{j}}^2}{{A_{j}}^2}$ for $1\leq j \leq i-1.$
\end{proposition}
\begin{proof}
By the choice of function $\psi$ and by $\|\cdot\|$ being induced by an inner product, we have that $D_{\psi^*}(\vw, \vz) = \frac{1}{2\mu}\|\vw - \vz\|^2.$ By the definition of $\vz_{k},$ we have that, $\forall i > j\geq 0$, $D_{\psi^*}(\vz_{i}, \vz_{j}) = \frac{1}{2\mu}\|\sum_{k = j+1}^{i} \frac{a_k}{A_k}H_k \nabla f(\vx_{k})\|^2.$ By Claim~\ref{claim:triangle-ineq}, $\forall i > j+1 >0:$
\begin{equation}\label{eq:pairs-of-breg-div}
    a_j D_{\psi^*}(\vz_{i}, \vz_{j}) + a_{j+1} D_{\psi^*}(\vz_{i}, \vz_{j+1}) \geq \frac{1}{2\mu}\frac{a_j a_{j+1}}{a_j + a_{j+1}}\Big\| \frac{a_{j+1}}{A_{j+1}}H_{j+1}\nabla f(\vx_{j+1}) \Big\|^2.
\end{equation}
Write $\sum_{j=0}^{i-1} a_j D_{\psi^*}(\vz_{i}, \vz_{j})$ as:
\begin{align*}
\sum_{j=0}^{i-1} a_j D_{\psi^*}(\vz_{i}, \vz_{j}) =&\ \frac{a_0}{2}D_{\psi^*}(\vz_{i}, \vz_{0})+ \frac{a_{i-1}}{2}D_{\psi^*}(\vz_{i}, \vz_{i-1})\\
&
+ \frac{1}{2}\sum_{j=0}^{i-2}\Big(a_j D_{\psi^*}(\vz_{i}, \vz_{j}) + a_{j+1} D_{\psi^*}(\vz_{i}, \vz_{j+1})\Big). 
\end{align*}
Combine the last equation with Eq.~\eqref{eq:pairs-of-breg-div} and $\frac{a_0}{2}D_{\psi^*}(\vz_{i}, \vz_{0}) \geq 0$. 
\end{proof}

\paragraph{Discrete-Time Methods}
The particular two-step discretization (reminiscent of a predic\-tor-corrector method) that we consider for general momentum dynamics~\eqref{eq:ct-mom-dyn} is:
\begin{equation}\tag{GMD}
    \begin{gathered}\label{eq:gen-mom-method}
        \vx_{k} = \frac{A_{k-1}}{A_{k}}\vy_{k-1} + \frac{a_k}{A_k}\nabla\psi^*(\vz_{k-1}),\\
        \vz_{k} = \vz_{k-1} - \frac{a_k}{A_k}H_k \nabla f(\vx_{k}),\\
        \vy_{k} = \vx_{k} + \frac{a_k}{A_k}\big(\nabla\psi^*(\vz_{k}) - \nabla\psi^*(\vz_{k-1})\big).
    \end{gathered}
\end{equation}
When $H_k = A_k,$ \eqref{eq:gen-mom-method} is equivalent to~\eqref{eq:gen-mom-method-f}, AGD+~\cite{cohen2018acceleration}, and the method of similar triangles~\cite{gasnikov2018universal}, which generalize Nesterov's accelerated method~\cite{Nesterov1983} and accelerated extra-gradient method~\cite{AXGD}. When $H_k = 1,$ \eqref{eq:gen-mom-method} is a slightly different discretization of the generalized heavy-ball dynamics than~\eqref{eq:gen-mom-method-f} with $H_k = 1$, where the difference lies only in the size of the extrapolation step $\vx_{k}.$ Working with the general momentum method~\eqref{eq:gen-mom-method} will allow us to obtain results for AGD+ and the generalized heavy-ball method as special cases, and it will also allow us to understand how the different choices of $H_k$ affect the convergence in norm of the gradient.
%
%
\paragraph{Discretization Error}
Characterization of the discretization error is what crucially determines the convergence of the methods in norm of the gradient, as well as in function value. Here, we show how the discretization error is affected by the choice of the step size and assumptions about the  objective function, such as smoothness and convexity.

\begin{lemma}\label{lemma:discretization-error}
Let $E_k \defeq \cc_{k} - \cc_{k-1},$ where $\cc_k$ was defined in~\eqref{eq:discr-cq}, with $a_i, b_i > 0,$ $\forall i$, $A_k = \sum_{i=0}^k a_i$, $B_k = \sum_{i=0}^k b_i,$ and $B_{k-1} = H_k A_{k-1}$. Let $\vx_{k}, \vy_{k}, \vz_{k}$ evolve according to~\eqref{eq:gen-mom-method}, where $\vx_{0} = \vy_{0}$ is an arbitrary initial point such that $\nabla\psi^*(\vz_{0}) = \vx_{0}$, $\psi: E \rightarrow \mathbb{R}$ is a $\mu$-strongly convex function w.r.t.~$\|\cdot\|$, and $f$ is an $L$-smooth function w.r.t.~$\|\cdot\|$.  If $f$ is $\epsilon_H$-weakly convex for $\epsilon_H \in [0, L]$ and $\frac{{a_k}^2}{{A_k}^2} = \frac{c\mu}{L H_k}$ for $c \in [0, 1],$ then 
$$
E_k \leq - (1-c)A_{k-1}D_{\psi^*}(\vz_{k-1}, \vz_{k}) + B_{k-1}\frac{\epsilon_H}{2}\|\vx_{k} - \vy_{k-1}\|^2.
$$
\end{lemma}
\begin{proof}
By the definitions of $E_k$ and $\cc_k$ and $D_{\psi^*}(\vz, \vz)=0,$ we have:
\begin{align}
    E_{k} = B_{k-1}\left[f(\vy_{k}) - f(\vy_{k-1})\right] + \sum_{i=0}^{k-1}a_i\left[D_{\psi^*}(\vz_{k}, \vz_{i}) - D_{\psi^*}(\vz_{k-1}, \vz_{i})\right]. \label{eq:discr-err-1}
\end{align}
As $f$ is $L$-smooth and $\epsilon_H$-weakly convex, we have:
\begin{align}\small
    f(\vy_{k}) - f(\vy_{k-1}) =&\ f(\vy_{k}) - f(\vx_{k}) + f(\vx_{k}) - f(\vy_{k-1})\notag\\
    \leq& \innp{\nabla f(\vx_{k}), \vy_{k} - \vy_{k-1}} + \frac{L}{2}\|\vy_{k} - \vx_{k}\|^2 + \frac{\epsilon_H}{2}\|\vx_{k} - \vy_{k-1}\|^2\notag\\
    =& \innp{\nabla f(\vx_{k}), \vy_{k} - \vy_{k-1}} + \frac{L}{2}\frac{{a_k}^2}{{A_k}^2}\|\nabla\psi^*(\vz_{k}) - \nabla\psi^*(\vz_{k-1})\|^2 \notag\\
    &+\frac{\epsilon_H}{2}\|\vx_{k} - \vy_{k-1}\|^2.\label{eq:discr-err-3}
\end{align}\normalsize

By the first part of Fact~\ref{fact:bregman-properties} and the definition of $\vz_{k}$, $\forall i \leq k-1$:
\begin{align*}
    D_{\psi^*}(\vz_{k-1}, \vz_{i}) =&\; D_{\psi^*}(\vz_{k}, \vz_{i}) + \innp{\nabla\psi^*(\vz_{k}) - \nabla\psi^*(\vz_{i}), \vz_{k-1} - \vz_{k}} + D_{\psi^*}(\vz_{k-1}, \vz_{k})\\
    =&\; D_{\psi^*}(\vz_{k}, \vz_{i}) + \frac{a_k}{A_k}H_k \innp{\nabla f(\vx_{k}), \nabla\psi^*(\vz_{k}) - \nabla\psi^*(\vz_{i})}\\
    &+ D_{\psi^*}(\vz_{k-1}, \vz_{k}).
\end{align*}
Hence, we have:
\begin{align}
    \sum_{i=0}^{k-1} a_i & \left[D_{\psi^*}(\vz_{k}, \vz_{i}) -  D_{\psi^*}(\vz_{k-1}, \vz_{i})\right]\notag\\
    =& -A_{k-1}D_{\psi^*}(\vz_{k-1}, \vz_{k})
    - \frac{a_k}{A_k}H_k \bigg\langle{\nabla f(\vx_{k}), A_{k-1}\nabla\psi^*(\vz_{k}) - \sum_{i=0}^{k-1}\nabla\psi^*(\vz_{i})}\bigg\rangle\notag\\
    =&  -A_{k-1}D_{\psi^*}(\vz_{k-1}, \vz_{k})
    - {a_k}{H_k}\bigg\langle{\nabla f(\vx_{k}), \nabla\psi^*(\vz_{k}) - \frac{1}{A_k}\sum_{i=0}^{k}\nabla\psi^*(\vz_{i})}\bigg\rangle\notag\\
    =& -A_{k-1}D_{\psi^*}(\vz_{k-1}, \vz_{k}) - {a_k}{H_k}\innp{\nabla f(\vx_{k}), \nabla\psi^*(\vz_{k}) -\vy_{k}}, \label{eq:discr-err-2}
\end{align}
where the last equality is by $\vy_{k} = \frac{1}{A_k}\sum_{i=0}^k a_i \nabla\psi^*(\vz_{i}),$ which follows by applying the definition of $\vy_{k}$ from~\eqref{eq:gen-mom-method} recursively and using $\vy_{0} = \nabla\psi^*(\vz_{0}).$

By Fact~\ref{fact:bregman-properties}, $D_{\psi^*}(\vz_{k-1}, \vz_{k})\geq \frac{\mu}{2} \|\nabla\psi^*(\vz_{k}) - \nabla\psi^*(\vz_{k-1})\|^2$. Hence, combining Eqs.~\eqref{eq:discr-err-1}--\eqref{eq:discr-err-2}:
\begin{align*}
    E_{k} \leq &\; B_{k-1}\Big\langle{\nabla f(\vx_{k}), \vy_{k} - \vy_{k-1} - \frac{a_k H_k}{B_{k-1}}(\nabla\psi^*(\vz_{k}) -\vy_{k}) }\Big\rangle +\frac{\epsilon_H B_{k-1}}{2}\|\vx_{k} - \vy_{k-1}\|^2\\
    &+ \Big(B_{k-1}\frac{L}{2}\frac{{a_k}^2}{{A_k}^2} - \frac{c\mu}{2}A_{k-1}\Big)\|\nabla\psi^*(\vz_{k}) - \nabla\psi^*(\vz_{k-1})\|^2\\
    &- (1-c)A_{k-1}D_{\psi^*}(\vz_{k-1}, \vz_{k})\\
    \leq&\; \frac{\epsilon_H B_{k-1}}{2}\|\vx_{k} - \vy_{k-1}\|^2 - (1-c)A_{k-1}D_{\psi^*}(\vz_{k-1}, \vz_{k}),
\end{align*}
where we have used $B_{k-1} = H_k A_{k-1},$ $\vy_{k} = \frac{A_{k-1}}{A_k}\vy_{k-1} + \frac{a_k}{A_{k-1}}\nabla\psi^*(\vz_{k})$ (which implies $\vy_{k} - \vy_{k-1} - \frac{a_k H_k}{B_{k-1}}(\nabla\psi^*(\vz_{k}) -\vy_{k}) = 0$), and $\frac{{a_k}^2}{{A_k}^2} = c \cdot \frac{\mu}{L H_k}.$
\end{proof}

\paragraph{Final Convergence Bound}

To be able to bound the non-negative term in the discretization error $E_k$ (which comes from $\epsilon_H$-weak convexity), we will make use of the following proposition.
\begin{proposition}\label{prop:ncvx-discr-err}
Let $\vx_{k}, \vy_{k}, \vz_{k}$ evolve according to~\eqref{eq:gen-mom-method}, for $\vx_{0} = \vy_{0} = \nabla\psi^*(\vz_{0})$ and $\mu$-strongly convex $\psi.$ Then:
$$
\frac{1}{2}\|\vx_{k} - \vy_{k-1}\|^2 \leq \frac{1}{\mu}\frac{{a_k}^2}{{A_k}^2 A_{k-1}}\sum_{i=0}^{k-2}a_i D_{\psi^*}(\vz_{k-1}, \vz_{i}). 
$$
\end{proposition}
\begin{proof}
 By applying the definition of $\vy_{k}$ recursively in~\eqref{eq:gen-mom-method}, we have that $\vy_{k} = \frac{1}{A_k}\sum_{i=0}^k a_i \nabla\psi^*(\vz_{i}),$ $\forall k \geq 0.$ Further, by the definition of $\vx_{k}$ in~\eqref{eq:gen-mom-method}, we have $\vx_{k} = \frac{A_{k-1}}{A_k}\vy_{k-1} + \frac{a_k}{A_k}\nabla\psi^*(\vz_{k-1}),$ and it follows that:
 \begin{align*}
     \vx_{k} - \vy_{k-1} &= \frac{a_k}{A_k}\big(\nabla\psi^*(\vz_{k-1}) - \vy_{k-1}\big)\\
     &= \frac{a_k}{A_k A_{k-1}}\sum_{i=0}^{k-1}a_i \big(\nabla\psi^*(\vz_{k-1}) - \nabla\psi^*(\vz_{i})\big).
 \end{align*}
 Applying Jensen's inequality:
 \begin{equation*}
     \|\vx_{k} - \vy_{k-1}\|^2 \leq \frac{{a_k}^2}{{A_k}^2A_{k-1}} \sum_{i=0}^{k-1}a_i\|\nabla\psi^*(\vz_{k-1}) - \nabla\psi^*(\vz_{i})\|^2.
 \end{equation*}
 The rest of the proof follows by using (by Fact~\ref{fact:bregman-properties}), $D_{\psi^*}(\vz, \vw) \geq \frac{\mu}{2}\|\nabla\psi^*(\vz) - \nabla\psi^*(\vw)\|^2,$ $\forall \vz, \vw.$ 
\end{proof}

Using Lemmas ~\ref{lemma:avg-fn-value-dec} and~\ref{lemma:discretization-error}, we now show how to bound the minimum norm of the gradient, under suitable step sizes, so that Lemma~\ref{lemma:discretization-error} applies. 

\begin{theorem}\label{thm:min-grad-norm}
Let $\vx_{k},\, \vy_{k},\, \vz_{k}$ evolve according to~\eqref{eq:gen-mom-method}, for arbitrary $\vx_{0} = \vy_{0}\in \mathbb{R}^n$ such that $\vx_{0} = \nabla\psi^*(\vz_{0}),$ where $\psi(\vx) = \frac{\mu}{2}\|\vx\|^2.$ Let $\frac{{a_i}^2}{{A_i}^2} = c \frac{\mu}{L H_i}$ for some $c \in [0, 1]$ and $B_{k-1} = A_{k-1}H_k.$ If, for some $c' \in [0, 1]$ and all $i \leq k$: 
$
(1-c')\big(\frac{1}{B_{i-1}}- \frac{1}{B_i}\big) \geq \frac{c\epsilon_H}{L}\big(\frac{1}{B_i} - \frac{1}{B_k}\big),
$
then:
\begin{align*}
   & c' \sum_{i=1}^k \bigg[\frac{1}{B_{i-1}} - \frac{1}{B_i}\bigg]\sum_{j=0}^{i-1}  a_j D_{\psi^*}(\vz_{i}, \vz_{j}) \; +\;   \frac{c(1-c)}{2L}\sum_{i=1}^{k} \Big(1 - \frac{B_{i-1}}{B_k}\Big) \|\nabla f(\vx_i)\|^2\\
    &\hspace{1.5in} \leq f(\vx_0) - f(\vx^*).
\end{align*}
\end{theorem}
\begin{proof}
Using Lemma~\ref{lemma:avg-fn-value-dec}, we have:
\begin{equation}\label{eq:main-identity}
\begin{aligned}
  &\sum_{i=1}^{k} \Big(\frac{1}{B_{i-1}} - \frac{1}{B_i}\Big) \sum_{j=0}^{i-1} a_j D_{\psi^*}(\vz_{i}, \vz_{j}) - \sum_{i=1}^{k}\Big(\frac{1}{B_{i-1}} - \frac{1}{B_k}\Big)E_i\\
  &\hspace{1in}= f(\vy_{0}) - \frac{1}{B_k}\sum_{i=0}^k b_i f(\vy_{i}) \leq f(\vx_0) - f(\vx^*),
\end{aligned}
\end{equation}
as $\vy_0 = \vx_0$ and $\vx^*$ minimizes $f.$ 

To prove the theorem, it suffices to bound from below the left-hand side of Eq.~\eqref{eq:main-identity}. 
Let us first bound the discretization error. Using Proposition~\ref{prop:ncvx-discr-err}, we have that, $\forall i:$
\begin{equation}\notag
\begin{aligned}
    E_i &\leq \frac{\epsilon_H B_{i-1}}{\mu}\frac{{a_{i}}^2}{{A_{i}}^2 A_{i-1}}\sum_{j=0}^{i-2}a_j D_{\psi^*}(\vz_{i-1}, \vz_{j}) - (1-c)A_{i-1}D_{\psi^*}(\vz_{i-1}, \vz_{i})\\
    &= c \frac{\epsilon_H}{L}\sum_{j=0}^{i-2}a_j D_{\psi^*}(\vz_{i-1}, \vz_{j}) - (1-c)A_{i-1}D_{\psi^*}(\vz_{i-1}, \vz_{i}).
\end{aligned}
\end{equation}
Therefore:
\begin{equation}
\begin{aligned}
    \sum_{i=1}^{k}\Big(\frac{1}{B_{i-1}} - \frac{1}{B_k}\Big)E_i \leq &\ c \frac{\epsilon_H}{L}\sum_{i=2}^{k}\Big(\frac{1}{B_{i-1}} - \frac{1}{B_k}\Big)\sum_{j=0}^{i-2}a_j D_{\psi^*}(\vz_{i-1}, \vz_{j}) \\
    &-  (1-c) \sum_{i=1}^{k}\Big(\frac{1}{B_{i-1}} - \frac{1}{B_k}\Big) A_{i-1}D_{\psi^*}(\vz_{i-1}, \vz_{i}).
\end{aligned}
\end{equation}
As $D_{\psi^*}(\vz_{i-1}\, \vz_i) = \frac{1}{2\mu}\|\vz_i - \vz_{i-1}\|^2 = \frac{1}{2\mu} \frac{{a_i}^2}{{A_i}^2}{H_i}^2\|\nabla f(\vx_i)\|^2,$ we further have:
\begin{equation}\label{eq:linear-in-k-term}
    \begin{aligned}
        \sum_{i=1}^{k}\Big(\frac{1}{B_{i-1}} - \frac{1}{B_k}\Big) A_{i-1}D_{\psi^*}(\vz_{i-1}, \vz_{i}) &=  \sum_{i=1}^{k}\Big(\frac{1}{B_{i-1}} - \frac{1}{B_k}\Big) A_{i-1}\frac{1}{2\mu} \frac{{a_i}^2}{{A_i}^2}{H_i}^2\|\nabla f(\vx_i)\|^2\\
        &= \frac{c}{2L} \sum_{i=1}^{k} \Big(1 - \frac{B_{i-1}}{B_k}\Big) \|\nabla f(\vx_i)\|^2,
    \end{aligned}
\end{equation}
where we have used $B_{i-1} = A_{i-1}H_i$ and $\frac{{a_i}^2}{{A_i}^2} = \frac{c}{H_i}\frac{\mu}{L},$ both from the statement of the theorem.

Combining Eqs.~\eqref{eq:main-identity}--\eqref{eq:linear-in-k-term}, we have:
\begin{align*}
    f(\vx_0) - f(\vx^*) \geq &\; \sum_{i=1}^k \bigg[\frac{1}{B_{i-1}} - \frac{1}{B_i} - \frac{c \epsilon_H}{L}\Big(\frac{1}{B_i} - \frac{1}{B_k}\Big)\bigg]\sum_{j=0}^{i-1} a_j D_{\psi^*}(\vz_{i}, \vz_{j})\\
    &+  \frac{c(1-c)}{2L}\sum_{i=1}^{k} \Big(1 - \frac{B_{i-1}}{B_k}\Big) \|\nabla f(\vx_i)\|^2.
\end{align*}
To complete the proof, it remains to use $(1-c')\big(\frac{1}{B_{i-1}}- \frac{1}{B_i}\big) \geq \frac{c\epsilon_H}{L}\big(\frac{1}{B_i} - \frac{1}{B_k}\big)$.
\end{proof}

To obtain useful convergence bounds, we need to show that it is possible to satisfy the assumptions of Theorem~\ref{thm:min-grad-norm}. We start by providing examples for the case of convex objectives, and then discuss the nonconvex case.

\paragraph{The Convex Case}\label{sec:cvx-opt}

When $f$ is convex, $\epsilon_H = 0,$ and Theorem~\ref{thm:min-grad-norm} can be applied with $c' = 0.$ Further, once $c \in [0, 1],$ $\mu,$ and $H_i$ are specified, all other parameters are set, since $a_i$ and $A_i$ can be computed from $\frac{{a_i}^2}{{A_i}^2} = \frac{c\mu}{H_i L}$ and $A_k = \sum_{i=0}^k a_i,$ and, finally, we have that $B_i = A_{i-1}H_i.$ The only restriction that Theorem~\ref{thm:min-grad-norm} imposes is that $B_i$ is a non-decreasing sequence.

To illustrate the results, we take $H_i = {A_i}^{\lambda},$ for $\lambda \in [0, 2].$ As mentioned before, $\lambda = 1$ corresponds to~AGD+~\cite{cohen2018acceleration} and $\lambda = 0$ corresponds to a generalization of the heavy-ball method. 
For this choice of $H_i,$ we have that $\frac{{a_i}^2}{{A_i}^2} = \frac{c\mu}{{A_i}^{\lambda}L},$ or, equivalently: 
$
\frac{{a_i}^2}{{A_i}^{2-\lambda}} = c\frac{\mu}{L}.
$ 
It is not hard to verify (using the asymptotic formula $\sum_{i=1}^k i^p = \frac{k^{p+1}}{p+1} + \frac{k^p}{2} + O(k^{p-1})$) that a sequence $\{a_i\}_i$ that satisfies this condition will grow as:
\begin{equation}\label{eq:ai-growth}
a_i \propto \begin{cases}
\big(\frac{c\mu}{L}\big)^{1/\lambda}i^{(2-\lambda)/\lambda}, \quad &\text{if } \lambda > 0 \text{ and } i = \Omega(\frac{1}{\lambda}),\\
\sqrt{\frac{c\mu}{L}}\big(1 - \sqrt{\frac{c\mu}{L}}\big)^{-(i-1)}, \quad &\text{if } \lambda = 0.
\end{cases}
\end{equation}
The following corollary (of Theorem~\ref{thm:min-grad-norm}) shows that any generalized momentum method~\eqref{eq:gen-mom-method} with $H_i = {A_i}^{\lambda},$ $\lambda \in [0, 2],$ converges to a point with small gradient norm at rate $1/k.$

\begin{restatable}{corollary}{hilbertcvx}\label{cor:gmd-convex}
Let $\vx_k,\, \vy_k,\, \vz_k$ evolve as in~\eqref{eq:gen-mom-method}, for convex f,  $\psi^*(\vz) = \frac{1}{2\mu}\|\vz\|^2,$ $0<\mu < L,$ and  $\frac{{a_i}^2}{{A_i}^{2-\lambda}} = c\frac{\mu}{L}$ for some $\lambda \in [0, 2]$, $c\in (0, 1],$ and $c\frac{\mu}{L} < 1$. Then, $\forall k \geq 1$:
$$
\min_{1 \leq i \leq k} \|\nabla f(\vx_i)\|^2 =  O\Big(\frac{L(f(\vx_0) - f(\vx^*))}{c(1-c)k + c \min\{\log k/\lambda, \, k\sqrt{c\mu/L}\} }\Big).
$$
In particular, for $c = \frac{1}{2}:$ 
$
\min_{1 \leq i \leq k} \|\nabla f(\vx_i)\|^2 =  O(\frac{L(f(\vx_0) - f(\vx^*))}{k}).
$
\end{restatable}
\begin{proof}
Applying Theorem~\ref{thm:min-grad-norm}, we have:
\begin{equation}\label{eq:gen-conv-bnd}
\begin{aligned}
 f(\vx_0) - f(\vx^*) \geq &\; 
\sum_{i=1}^k \bigg[\frac{1}{B_{i-1}} - \frac{1}{B_i}\bigg]\sum_{j=0}^{i-1} a_j D_{\psi^*}(\vz_{i}, \vz_{j})\\
&+   \frac{c(1-c)}{2L}\sum_{i=1}^{k} \Big(1 - \frac{B_{i-1}}{B_k}\Big) \|\nabla f(\vx_i)\|^2.
\end{aligned}
\end{equation}
Consider first the case $\lambda \in (0, 2].$ Then, from Eq.~\eqref{eq:ai-growth}, we have $a_i \propto \big(\frac{c\mu}{L}\big)^{1/\lambda}i^{(2-\lambda)/\lambda}$ and $A_i \propto \frac{\lambda}{2}\big(\frac{c\mu}{L}\big)^{1/\lambda}i^{2/\lambda}$ for $i = \Omega(\frac{1}{\lambda})$, in which case also 
$$
B_i = \Omega\Big(\lambda^{1+\lambda} \big(\frac{c\mu}{L}\big)^{(1+\lambda)/\lambda}\, i^{\frac{2(1+\lambda)}{\lambda}  }\Big).
$$ 
When $\lambda = 0,$ we have:
$
B_i = A_{i-1} \propto \Big(1 - \sqrt{\frac{c\mu}{L}}\Big)^{-(i-1)}.
$ 
In either case:
$$
\frac{a_{j-1}}{A_{j-1}} = \Omega\Big( \min\Big\{\sqrt{\frac{c\mu}{L}},\,\frac{1}{\lambda j}\Big\}\Big).
$$

As $B_i$ grows as a  function of $i$ at least cubically, we have that: 
$$
\frac{c(1-c)}{2L}\sum_{i=1}^{k} \Big(1 - \frac{B_{i-1}}{B_k}\Big) = \frac{c(1-c)}{2L}\Theta(k).
$$
It remains to bound $\sum_{i=1}^k \big[\frac{1}{B_{i-1}} - \frac{1}{B_i}\big]\sum_{j=0}^{i-1} a_j D_{\psi^*}(\vz_{i}, \vz_{j}).$ By Proposition~\ref{prop:BD-to-grad-norms} and $a_j \geq a_{j-1}$, 
\begin{align*}
    \sum_{j=0}^{i-1} a_j D_{\psi^*}(\vz_{i}, \vz_{j}) &\geq \frac{1}{2\mu}\sum_{j=1}^i \frac{a_{j-1}}{2}\frac{{a_j}^2}{{A_j}^2}{H_j}^2 \|\nabla f(\vx_j)\|^2 
    = \frac{c}{4L} \sum_{j=1}^i \frac{a_{j-1}}{A_{j-1}} B_{j-1} \|\nabla f(\vx_j)\|^2.
\end{align*}

As $\frac{a_{j-1}}{A_{j-1}}=\Omega( \min\{\sqrt{\frac{c\mu}{L}},\,\frac{1}{\lambda j}\}),$ we have that:
\begin{align*}
 &\sum_{i=1}^k \bigg[\frac{1}{B_{i-1}} - \frac{1}{B_i}\bigg]\sum_{j=0}^{i-1} a_j D_{\psi^*}(\vz_{i}, \vz_{j})\\
 &\hspace{.7in}= \Omega\Big(\frac{c}{ L}\Big)\min_{1\leq i\leq k}\|\nabla f(\vx_i)\|^2   \sum_{i=1}^k \bigg[\frac{1}{B_{i-1}} - \frac{1}{B_i}\bigg] \sum_{j=1}^i \min\Big\{\sqrt{\frac{c\mu}{L}}, \, \frac{1}{\lambda j}\Big\} B_{j-1}.
\end{align*}
Finally, observe that:
\begin{align*}
    \sum_{i=1}^k \bigg[\frac{1}{B_{i-1}} - \frac{1}{B_i}\bigg] \sum_{j=1}^i \min & \Big\{\sqrt{\frac{c\mu}{L}}, \, \frac{1}{\lambda j}\Big\} B_{j-1}\\
    &= \sum_{j=1}^k \Omega\Big(\min\Big\{\sqrt{\frac{c\mu}{L}}, \, \frac{1}{\lambda j}\Big\}\Big) B_{j-1}\sum_{i=j}^k \bigg[\frac{1}{B_{i-1}} - \frac{1}{B_i}\bigg]\\
    &= \sum_{j=1}^k \Omega\Big(\min\Big\{\sqrt{\frac{c\mu}{L}}, \, \frac{1}{\lambda j} \Big\}\Big)\bigg[1 - \frac{B_{j-1}}{B_k}\bigg]\\
    &=\Omega\Big(\min\big\{\log(k)/\lambda,\, k \sqrt{{c\mu}/{L}}\big\}\Big).
\end{align*}
Hence, we obtain:
$$
\min_{1 \leq i \leq k}\|\nabla f(\vx_i)\|^2  = O\Big(\frac{L(f(\vx_0) - f(\vx^*))}{c(1-c)k + c \min\{\log k/\lambda, \, k\sqrt{c\mu/L}\} }\Big),
$$
as claimed.
\end{proof}
A few remarks are in order here. It is not hard to see that for an~\emph{arbitrary} positive sequence of numbers $a_i,\, H_i,\, B_i$ that satisfy $\frac{{a_i}^2}{{A_i}^2} = c\frac{\mu}{L H_i}$, $A_i = \sum_{j=0}^i a_j,$ and $B_{i-1} = A_{i-1}H_i$, it is not possible to get better than a $1/k$ rate for convergence to stationary points, as long as Proposition~\ref{prop:BD-to-grad-norms} is used. This rate is known to be suboptimal---the optimal rate for smooth convex functions is $1/k^2$ and it is achieved by the OGM-G algorithm from~\cite{kim2018optimizing}. The rate $1/k$ for the generalized momentum methods is not surprising, and in the case of $\lambda = 1$ ($H_i = A_i,$ in which case the method is essentially equivalent to Nesterov's accelerated method in Euclidean spaces), this rate is known to be tight~\cite{kim2018generalizing}. We expect that the same is true for an arbitrary (but fixed) value of $\lambda,$ though this may be possible to show only numerically~\cite{kim2018generalizing}.

%
%
\paragraph{The Nonconvex Case}

The main restriction for obtaining the results in the nonconvex case is ensuring that:
$$
(1-c')\Big(\frac{1}{B_{i-1}}- \frac{1}{B_i}\Big) \geq \frac{c\epsilon_H}{L}\Big(\frac{1}{B_i} - \frac{1}{B_k}\Big),
$$
for some $c, c'.$ Let us first study what kind of a constraint on the sequence $B_i$ such a condition imposes. Rearranging the terms in the last expression, we have that:
\begin{equation}\label{eq:condition-c'}
\frac{B_i}{B_{i-1}} \geq 1 + \frac{c\epsilon_H}{(1-c')L}\Big(1 - \frac{B_i}{B_k}\Big).
\end{equation}
Since the last expression needs to be satisfied for all $i,$ when $B_i$ grows polynomially with $i,$ it is not hard to verify that to ensure $c' \geq 0,$ we would need to have $\frac{c \epsilon_H}{L} = O(\frac{1}{i}),$ $\forall i$; that is, for a fixed number of iterations $k,$ we would need $\frac{c \epsilon_H}{L} = O(\frac{1}{k})$.  This leads to uninformative convergence bounds, unless $\epsilon_H = O(L/k)$ (in which case one can show that $\min_{1\leq i \leq k}\|\nabla f(\vx_i)\|^2 = O(\frac{L(f(\vx_0) - f(\vx^*))}{c(1-c)k})$).

When $\epsilon_H = \omega(L/k),$ we are unable to show the convergence rate of $1/k$ for polynomially growing sequences $a_i$ (and, consequently, polynomially growing $A_i,$ $B_i$). Instead, we can only show such a convergence rate for constant $H_i.$ In particular, let us choose $\mu$ and $H_i$ so that $\frac{\mu}{L H_i} = 1.$ Then $\frac{a_i}{A_i} = \sqrt{c},$ and it follows that $\frac{B_i}{B_{i-1}} = (1-\sqrt{c})^{-1}.$ As $\epsilon_H \leq L$ and $1- \frac{B_i}{B_k}\leq 1,$ to satisfy the condition from Eq.~\eqref{eq:condition-c'}, it suffices to have $\sqrt{c} \geq \frac{c}{1-c'}.$ Equivalently, Eq.~\eqref{eq:condition-c'} is satisfied with: 
$$
c' \leq 1- \sqrt{c}.
$$
By the same arguments as in the proof of Corollary~\ref{cor:gmd-convex}, we immediately have:
\begin{restatable}{corollary}{hilbertncvx}\label{cor:gmd-nonconvex}
Let $\vx_k,\, \vy_k,\, \vz_k$ evolve according to~\eqref{eq:gen-mom-method}, for $\psi^*(\vz) = \frac{1}{2\mu}\|\vz\|^2,$ $\mu > 0,$ $\frac{\mu }{L H_i}=1$,  $\frac{{a_i}^2}{{A_i}^{2}} = c$, and $c\in (0, 1)$. Then, $\forall k\geq 1$:
$$
\min_{1 \leq i \leq k} \|\nabla f(\vx_i)\|^2 =  O\Big(\frac{L(f(\vx_0) - f(\vx^*))}{c(1-c)k + \sqrt{c}(1-\sqrt{c}) c k}\Big).
$$
In particular, for $c = \frac{1}{2}$: 
$
\min_{1 \leq i \leq k} \|\nabla f(\vx_i)\|^2 =  O\big(\frac{L(f(\vx_0) - f(\vx^*))}{k}\big).
$
\end{restatable}
\subsection{Convergence to Stationary Points in  \markupadd{non-Euclidean} Spaces}\label{sec:banach}

We now show that it is possible to obtain results for convergence to stationary points even in \markupdelete{Banach} \markupadd{non-Euclidean} spaces $(E, \|\cdot\|)$. 
We are only able to show such a result, however, for a different discretization of~\eqref{eq:ct-mom-dyn}. To obtain the result, we require only that $\psi(\cdot)$ is $\mu$-strongly convex with respect to the norm $\|\cdot\|.$ By Fact~\ref{fact:smoothness-sc-duality}, this implies that $\psi^*$ is $\frac{1}{\mu}$-smooth with respect to the dual norm $\|\cdot\|_*.$ 

The alternative discretization uses a gradient descent step for $\vy_k$ to ensure a decrease in $\cc_k$ that depends on $\|\nabla f(\vy_k)\|_*^2.$ However, to ensure the right change in $\cc_k$ over iterations $k,$ such a choice of $\vy_k$ requires changes to the extrapolation step $\vx_k.$ In particular, the discrete-time algorithm is:
\begin{equation}\label{eq:gen-mom-method-banach}\tag{GMD\textsubscript{B}}
    \begin{gathered}
        \vx_k = \vy_{k-1} + \frac{a_k}{A_k} \nabla \psi^*(\vz_{k-1}) - \frac{a_k}{A_k A_{k-1}}\sum_{i=0}^{k-1}a_i \nabla \psi^*(\vz_i),\\
        \vz_k = \vz_{k-1} - \frac{a_k}{A_k}H_k \nabla f(\vx_k),\\
        \vy_k = \argmin_{\vu} \Big\{\innp{\nabla f(\vx_k), \vu - \vx_k} + \frac{L}{2}\|\vu - \vx_k\|^2\Big\}.
    \end{gathered}
\end{equation}
Note that in \markupdelete{Hilbert} \markupadd{Euclidean} spaces, when $\psi^*(\vz) = \frac{\mu}{2}\|\vz\|^2,$ \eqref{eq:gen-mom-method-banach} is equivalent to~\eqref{eq:gen-mom-method}. Hence, \eqref{eq:gen-mom-method-banach} can be seen as a generalization of~\eqref{eq:gen-mom-method} to \markupdelete{Banach} \markupadd{non-Euclidean} spaces.

\paragraph{Discretization Error} 
As in the previous section, we start by bounding the discretization error $E_k \defeq \cc_k - \cc_{k-1}.$ 

\begin{lemma}\label{lemma:discretization-error-banach}
Let $E_k \defeq \cc_{k} - \cc_{k-1},$ where $\cc_k$ was defined in~\eqref{eq:discr-cq}, with $a_i, b_i > 0,$ $\forall i$, $A_k = \sum_{i=0}^k a_i$, $B_k = \sum_{i=0}^k b_i,$ and $B_{k-1} = H_k A_{k-1}$. Let $\vx_{k},\, \vy_{k},\, \vz_{k}$ evolve according to~\eqref{eq:gen-mom-method-banach}, where $\vx_{0} = \vy_{0}\in E$ is an arbitrary initial point such that $\nabla\psi^*(\vz_{0}) = \vx_{0}$, $\psi: E \rightarrow \mathbb{R}$ is a $\mu$-strongly convex function w.r.t.~$\|\cdot\|$, and $f$ is an $L$-smooth function w.r.t.~$\|\cdot\|$.  If $f$ is $\epsilon_H$-weakly convex for $\epsilon_H \in [0, L]$ and $\frac{{a_k}^2}{{A_k}^2} =  \frac{c\mu}{L H_k}$ for $c \in [0, 1],$ then 
$$
E_k \leq - (1-c)\frac{B_{k-1}}{2L} \|\nabla f(\vx_k)\|_*^2 + B_{k-1}\frac{\epsilon_H}{2}\|\vx_{k} - \vy_{k-1}\|^2.
$$
\end{lemma}
\begin{proof}
The proof uses similar arguments as the proof of Lemma~\ref{lemma:discretization-error}.
By the definitions of $E_k$ and $\cc_k$ and $D_{\psi^*}(\vz, \vz)=0,$ we have:
\begin{align}
    E_{k} = B_{k-1}\left[f(\vy_{k}) - f(\vy_{k-1})\right] + \sum_{i=0}^{k-1}a_i\left[D_{\psi^*}(\vz_{k}, \vz_{i}) - D_{\psi^*}(\vz_{k-1}, \vz_{i})\right]. \label{eq:discr-err-1-b}
\end{align}
As $f$ is $L$-smooth and $\epsilon_H$-weakly convex, we have:
\begin{align}
    f(\vy_{k}) - f(\vy_{k-1}) =&\ f(\vy_{k}) - f(\vx_{k}) + f(\vx_{k}) - f(\vy_{k-1})\notag\\
    \leq & \innp{\nabla f(\vx_{k}), \vx_{k} - \vy_{k-1}} - \frac{1}{2L}\|\nabla f(\vx_k)\|_*^2 + \frac{\epsilon_H}{2}\|\vx_{k} - \vy_{k-1}\|^2.\label{eq:discr-err-3-b} 
\end{align}

By the first part of Fact~\ref{fact:bregman-properties} and the definition of $\vz_{k}$, $\forall i \leq k-1$:
\begin{align*}
    D_{\psi^*}(\vz_{k}, \vz_{i}) =&\; D_{\psi^*}(\vz_{k-1}, \vz_{i}) + \innp{\nabla\psi^*(\vz_{k-1}) - \nabla\psi^*(\vz_{i}), \vz_{k} - \vz_{k-1}} + D_{\psi^*}(\vz_{k}, \vz_{k-1})\\
    =&\; D_{\psi^*}(\vz_{k-1}, \vz_{i}) - \frac{a_k}{A_k}H_k \innp{\nabla f(\vx_{k}), \nabla\psi^*(\vz_{k-1}) - \nabla\psi^*(\vz_{i})}\\
    &+ D_{\psi^*}(\vz_{k}, \vz_{k-1}). 
\end{align*}
Hence, we have:
\begin{align}
    \sum_{i=0}^{k-1}& a_i\left[D_{\psi^*}(\vz_{k}, \vz_{i}) - D_{\psi^*}(\vz_{k-1}, \vz_{i})\right]\notag\\
    &= A_{k-1}D_{\psi^*}(\vz_{k}, \vz_{k-1})
    - \frac{a_k}{A_k}H_k \Big\langle{\nabla f(\vx_{k}), A_{k-1}\nabla\psi^*(\vz_{k-1}) - \sum_{i=0}^{k-1}\nabla\psi^*(\vz_{i})}\Big\rangle\notag\\
    &= A_{k-1}D_{\psi^*}(\vz_{k}, \vz_{k-1}) - B_{k-1}\innp{\nabla f(\vx_{k}), \vx_k - \vy_{k-1}}, \label{eq:discr-err-2-b}
\end{align}
where the last line is by $B_k = A_{k-1}H_k$ and the definition of $\vx_k,$ which can implies
$$\vx_k - \vy_{k-1} = \frac{a_k}{A_k A_{k-1}}(A_{k-1}\nabla \psi^*(\vz_{k-1}) - \sum_i^{k-1}a_i \nabla \psi^*(\vz_i)).$$ 
By Fact~\ref{fact:smoothness-sc-duality}, 
$D_{\psi^*}(\vz_{k}, \vz_{k-1})\leq \frac{1}{2\mu} \|\vz_{k} - \vz_{k-1}\|_*^2 = \frac{1}{2\mu} \frac{{a_k}^2{H_k}^2}{{A_k}^2}\|\nabla f(\vx_k)\|_*^2.$ 
Hence, combining Eqs.\eqref{eq:discr-err-1-b}--\eqref{eq:discr-err-2-b}:
\begin{align*}
    E_{k} \leq &\ B_{k-1}\Big(-\frac{1}{2L} + \frac{{a_k}^2}{2\mu{A_{k}}^2}H_k\Big)\|\nabla f(\vx_k)\|_*^2 + B_{k-1}\frac{\epsilon_H}{2}\|\vx_{k} - \vy_{k-1}\|^2\\
    =& -\frac{(1-c)B_{k-1}}{2L}\|\nabla f(\vx_k)\|_*^2 + B_{k-1}\frac{\epsilon_H}{2}\|\vx_{k} - \vy_{k-1}\|^2,
\end{align*}
where we have used that $B_{k-1} = H_k A_{k-1}$ and $\frac{{a_k}^2}{{A_k}^2} = c \cdot \frac{\mu}{L H_k}.$
\end{proof}

\paragraph{Final Convergence Bound}

Since the proof of Proposition~\ref{prop:ncvx-discr-err} only required that $\psi$ be $\mu$-strongly convex and that $\vx_k - \vy_{k-1} = \frac{a_k}{A_k}(\nabla \psi^*(\vz_{k-1})- \frac{1}{A_{k-1}}\sum_{i=0}^{k-1}a_i \nabla \psi^*(\vz_i)),$ the same claim holds for the iterates of~\eqref{eq:gen-mom-method-banach}. Thus, we can draw a similar conclusion as for~\eqref{eq:gen-mom-method}.

\begin{theorem}\label{thm:conv-banach}
Let $\vx_{k},\, \vy_{k},\, \vz_{k}$ evolve according to~\eqref{eq:gen-mom-method-banach}, for arbitrary $\vx_{0} = \vy_{0}\in E$ such that $\vx_{0} = \nabla\psi^*(\vz_{0}),$ where $\psi(\vx)$ is strongly convex w.r.t.~$\|\cdot\|.$ Let $\frac{{a_i}^2}{{A_i}^2} = c \frac{\mu}{L H_i}$ for some $c \in [0, 1]$ and $B_{k-1} = A_{k-1}H_k.$ If for some $c' \in [0, 1]$ and all $i \leq k$:
$
(1-c')\big(\frac{1}{B_{i-1}}- \frac{1}{B_i}\big) \geq \frac{c\epsilon_H}{L}\big(\frac{1}{B_i} - \frac{1}{B_k}\big),
$
then:
\begin{align*}
    &c'\sum_{i=1}^k \Big[\frac{1}{B_{i-1}} - \frac{1}{B_i}\Big]\sum_{j=0}^{i-1} a_j D_{\psi^*}(\vz_{i}, \vz_{j}) \; +\;   \frac{c(1-c)}{2L}\sum_{i=1}^{k} \Big(1 - \frac{B_{i-1}}{B_k}\Big) \|\nabla f(\vx_i)\|_*^2\\
    &\hspace{1.5in} \leq f(\vx_0) - f(\vx^*).
\end{align*}
\end{theorem}
The proof is the same as the proof of Theorem~\ref{thm:min-grad-norm} and is thus omitted.

The main difference between~\eqref{eq:gen-mom-method} and~\eqref{eq:gen-mom-method-banach} in terms of the conclusions about the convergence to stationary points is that, because we are no longer assuming strong convexity of $\psi^*,$ we can no longer bound $\sum_{j=0}^{i-1} a_j D_{\psi^*}(\vz_{i}, \vz_{j})$ below as a function of the norms of the gradients. However, the term $\frac{c(1-c)}{2L}\sum_{i=1}^{k} (1 - \frac{B_{i-1}}{B_k}) \|\nabla f(\vx_i)\|_*^2$ from the theorem statement is still sufficient for obtaining $1/k$ asymptotic convergence, as shown in the following corollary.

\begin{restatable}{corollary}{corbanach}\label{cor:banach-1/k-conv}
Let $\vx_k,\, \vy_k,\, \vz_k$ evolve as in~\eqref{eq:gen-mom-method-banach}, for some
$\mu$-strongly convex $\psi,$ where $\mu > 0,$ and where $\frac{{a_i}^2}{{A_i}^{2}} = c\frac{\mu}{L H_i}$, $c\in (0, 1)$. 
\begin{itemize}
\item[(i)] If $\frac{\mu}{L H_i}=1$, then $\forall k\geq 1$:
$
\min_{1 \leq i \leq k} \|\nabla f(\vx_i)\|_*^2 =  O\big(\frac{L(f(\vx_0) - f(\vx^*))}{c(1-c)k}\big).
$

In particular, for $c = \frac{1}{2}$:
$
\min_{1 \leq i \leq k} \|\nabla f(\vx_i)\|_*^2 =  O\big(\frac{L(f(\vx_0) - f(\vx^*))}{k}\big).
$
\item[(ii)] If $f$ is convex and $H_i = {A_i}^{\lambda},$ then:
$
\min_{1 \leq i \leq k} \|\nabla f(\vx_i)\|_*^2 =  O\big(\frac{L(f(\vx_0) - f(\vx^*))}{c(1-c)k}\big).
$

In particular, for $c = \frac{1}{2}$:
$
\min_{1 \leq i \leq k} \|\nabla f(\vx_i)\|_*^2 =  O\big(\frac{L(f(\vx_0) - f(\vx^*))}{k}\big).
$
\end{itemize}
\end{restatable}
\begin{proof}
The first part of the corollary follows because under the assumption that $\frac{\mu L}{H_i}=1$, the condition of Theorem~\ref{thm:conv-banach} can be satisfied with $c' = 1 - \sqrt{c}\geq 0,$ as discussed in the previous subsection. Further, in this case $B_i$ grows exponentially fast and $\sum_{i=1}^{k} \big(1 - \frac{B_{i-1}}{B_k}\big) = \Omega(k).$ As $B_i$ is increasing and Bregman divergences are non-negative, we have:
$$
\frac{c(1-c)}{2L}\sum_{i=1}^{k} \Big(1 - \frac{B_{i-1}}{B_k}\Big) \|\nabla f(\vx_i)\|_*^2 \leq f(\vx_0) - f(\vx^*),
$$
which implies the claimed statement.

For the second part of the corollary, convexity of $f$ implies that the condition from Theorem~\ref{thm:conv-banach} can be satisfied with $c' = 1.$ As discussed in the proof of Corollary~\ref{cor:gmd-convex}, for $H_i = {A_i}^{\lambda},$ $B_i$ grows at least cubically in $i,$ which, again, implies that $\sum_{i=1}^{k} (1 - \frac{B_{i-1}}{B_k}) = \Omega(k),$ and leads to the same conclusion as in the first part of the proof.
\end{proof}
%
%
\section{Conclusion}
We presented a generic Hamiltonian-based framework for the analysis of general momentum methods in \markupdelete{Banach} \markupadd{non-Euclidean} spaces and in the settings of both convex and nonconvex optimization. Several questions that merit further investigation remain. For example, while convergence to stationary points in Euclidean spaces is well-understood~\cite{carmon2017lower,kim2018optimizing}, much less is known in terms of both upper and lower bounds in general \markupdelete{Banach} \markupadd{non-Euclidean} spaces. Another interesting direction for future research is a rigorous characterization of the use of momentum to escape shallow local minima. Finally, it is worth noting that, even though it has led to intuitive interpretations and fruitful results in classical and more recent optimization literature, the continuous-time perspective in optimization is not a panacea and certain phenomena are crucially discrete (see~\cite{ascher2019discrete} for a stimulating discussion on this topic). Thus, it is interesting to also understand the limitations of the continuous-time perspective.  
\section*{Acknowledgements} We thank Guilherme Fran\c{c}a, Michael Muehlebach, and Uri Ascher for useful comments and suggestions.  
%
%
\bibliographystyle{siamplain}
\bibliography{references}
%
%
\appendix
\section{Proof of Lemma~\ref{lemma:ct-gen-mom-cq}}
\label{appx:ct-proofs}
%
%
%
%
%
%
 The simplest way of proving the lemma is by directly computing $\frac{\dd}{\dd t}\cc_t^f$ and $\frac{\dd}{\dd t}\cc_t$, and showing that~\eqref{eq:ct-mom-dyn} implies that both are equal to zero. Here, we provide a longer, but more constructive proof that highlights how $\cc_t^f, \, \cc_t$ arise as invariants of~\eqref{eq:momentum-ham}.
 
 As $\overline{\vx}_t,\, \vz_t$ evolve according to the equations of motion of~\eqref{eq:momentum-ham}, we have that $\frac{\dd}{\dd t}\chm (\overline{\vx}_t, \vz_t, \alpha_t) = \frac{\partial}{\partial t}\chm(\overline{\vx}_t, \vz_t, \alpha_t) = \frac{\dd }{\dd \alpha_t}\chm(\overline{\vx}_t, \vz_t, \alpha_t) \cdot \dot{\alpha}_t.$ 
 Observe that: 
\begin{align*}
\frac{\dd}{\dd \alpha_t}\chm(\overline{\vx}_t, \vz_t, \alpha_t) &= h'(\alpha_t) f(\overline{\vx}_t/\alpha_t) - h(\alpha_t)\innp{\nabla f(\overline{\vx}_t/\alpha_t), \frac{\overline{\vx}_t}{{\alpha_t}^2}}\\
&= h'(\alpha_t)f(\vx_t) - \frac{h(\alpha_t)}{\alpha_t}\innp{\nabla f(\vx_t), \vx_t}.  
\end{align*}
Hence, we have that: 
\begin{align*}
\frac{\dd}{\dd t}\left(h(\alpha_t)f(\vx_t)+\psi^*(\vz_t)\right) &= \frac{\dd h(\alpha_t)}{\dd t} f(\vx_t) - h(\alpha_t)\frac{\dot{\alpha}_t}{\alpha_t} \innp{\nabla f(\vx_t), \vx_t}.
\end{align*}
Equivalently, using the product rule of differentiation:
\begin{equation}\label{eq:gen-invariant}
    h(\alpha_t)\frac{\dd}{\dd t}f(\vx_t) + \frac{\dd}{\dd t}\psi^*(\vz_t) = - h(\alpha_t)\frac{\dot{\alpha}_t}{\alpha_t}\innp{\nabla f(\vx_t), \vx_t}. 
\end{equation}
Integrating both sides of~\eqref{eq:gen-invariant} from 0 to $t$ and using integration by parts leads to $\cc_t^f.$ 

To obtain $\cc_t,$ observe (from~\eqref{eq:ct-mom-dyn}) that $\dot{\vz}_t = - h(\alpha_t)\frac{\dot{\alpha}_t}{\alpha_t}\nabla f(\vx_t)$. Multiplying both sides of~\eqref{eq:gen-invariant} by $\alpha_t$, we thus have:
\begin{equation}\label{eq:ct-diff-form}
    \alpha_t h(\alpha_t)\frac{\dd}{\dd t}f(\vx_t) + \alpha_t \frac{\dd}{\dd t}\psi^*(\vz_t) =  \alpha_t\innp{\dot{\vz}_t, \vx_t}.
\end{equation}
As for $\cc_t^f,$ to obtain $\cc_t,$ we integrate both sides of the last equation from 0 to $t$. Integrating the left-hand side and applying integration by parts gives:
\begin{equation}\label{eq:integral-of-lhs}
    \begin{aligned}
        \int_0^t \Big( \alpha_{\tau}  h(\alpha_{\tau}) \frac{\dd}{\dd {\tau}}f(\vx_{\tau}) +  \alpha_{\tau}& \frac{\dd}{\dd {\tau}}\psi^*(\vz_{\tau}) \Big)\dd\tau\\
        = &\; h(\alpha_t)\alpha_t f(\vx_t) - h(\alpha_0)\alpha_0 f(\vx_0) - \int_0^t \frac{\dd (h(\alpha_{\tau})\alpha_{\tau})}{\dd\tau} f(\vx_{\tau})\dd \tau \\
        &+ \alpha_t \psi^*(\vz_t) - \alpha_0 \psi^*(\vz_0) - \int_0^t \dot{\alpha}_{\tau} \psi^*(\vz_{\tau}) \dd\tau.
    \end{aligned}
\end{equation}
On the other hand, by the definition of $\vx_t$ in~\eqref{eq:ct-mom-dyn}, $\vx_t = \frac{\alpha_0}{\alpha_t}\vx_0 + \frac{1}{\alpha_t}\int_0^t \dot{\alpha_{\sigma}}\nabla \psi^*(\vz_{\sigma})\dd\sigma.$ Thus, integrating the right-hand side of~\eqref{eq:ct-diff-form}, we have:
\begin{equation}\label{eq:integral-of-rhs-1}
    \begin{aligned}
        \int_0^t \alpha_\tau \innp{\dot{\vz}_{\tau}, \vx_{\tau}}\dd\tau &= \int_0^t  \innp{ \dot{\vz}_{\tau},\, \alpha_0 \vx_0 + \int_0^{\tau} \nabla \psi^*(\vz_{\sigma})\dot{\alpha}_{\sigma}\dd\sigma} \dd\tau\\
        &= \alpha_0\innp{\vz_t - \vz_0, \nabla\psi^*(\vz_0)} + \int_0^t \int_0^{\tau} \innp{ \dot{\vz}_{\tau},\,  \nabla \psi^*(\vz_{\sigma})}\dot{\alpha}_{\sigma}\dd\sigma \dd\tau,
    \end{aligned}
\end{equation}
where we have used $\vx_0 = \nabla \psi^*(\vz_0).$ By elementary calculus, it is possible to exchange the order of integration on the right-hand side of~\eqref{eq:integral-of-rhs-1}, which leads to:
\begin{equation}\label{eq:integral-of-rhs-2}
    \begin{aligned}
        \int_0^t \alpha_\tau \innp{\dot{\vz}_{\tau}, \vx_{\tau}}\dd\tau &= \int_0^t  \innp{ \dot{\vz}_{\tau},\, \alpha_0 \vx_0 + \int_0^{\tau} \nabla \psi^*(\vz_{\sigma})\dot{\alpha}_{\sigma}\dd\sigma} \dd\tau\\
        &= \alpha_0\innp{\vz_t - \vz_0, \nabla\psi^*(\vz_0)} + \int_0^t \innp{ {\vz}_{t} - \vz_{\sigma},\,  \nabla \psi^*(\vz_{\sigma})}\dot{\alpha}_{\sigma}\dd\sigma.
    \end{aligned}
\end{equation}
By the definition of Bregman divergence, $D_{\psi*}(\vz, \vw) = \psi^*(\vz) - \psi^*(\vw) - \nabla \psi^*(\vw, \vz - \vw).$ Thus, combining Eqs.~\eqref{eq:integral-of-lhs} and~\eqref{eq:integral-of-rhs-2} leads to $\cc_t = 0.$ As this holds for an arbitrary $t,$ the proof is complete.

\end{document}